\documentclass[11pt]{article}
\usepackage{amsmath,amsthm,amssymb,amscd}
\usepackage{hyperref}
\usepackage{geometry}
\usepackage{tikz-cd}
\usepackage{booktabs} 
\usepackage{graphicx} 
\usepackage{multirow}
\usepackage{tikz}
\usetikzlibrary{calc,positioning,3d}
\usepackage[normalem]{ulem} 
\usepackage{imakeidx}
\makeindex

\usepackage{subfigure} 

\newtheorem{theorem}{Theorem}[section]
\newtheorem{lemma}[theorem]{Lemma}

\theoremstyle{definition}
\newtheorem{definition}[theorem]{Definition}
\newtheorem{example}[theorem]{Example}

\title{Persistent Stanley–Reisner Theory}
	\author{Faisal Suwayyid\footnote{Corresponding author: Faisal Suwayyid (faisal.suwayyid@kfupm.edu.sa).~}$^{~1,2}$
		and Guo-Wei Wei\footnote{Corresponding author: Guo-Wei Wei (weig@msu.edu).~}$^{~2,3,4}$ \\
		$^1$Department of Mathematics,\\
		King Fahd University of Petroleum and Minerals, Dhahran 31261, KSA.\\
		$^2$Department of Mathematics,\\
		Michigan State University, MI 48824, USA.\\
		$^3$Department of Electrical and Computer Engineering,\\
		Michigan State University, MI 48824, USA.\\
		$^4$Department of Biochemistry and Molecular Biology,\\
		Michigan State University, MI 48824, USA.
	}

	\date{\today}
\begin{document}
	\maketitle 
	\begin{abstract}	    
Topological data analysis (TDA) has emerged as an effective approach in data science, with its key technique, persistent homology, rooted in algebraic topology. Although alternative approaches based on differential topology, geometric topology, and combinatorial Laplacians have been proposed, combinatorial commutative algebra has hardly been developed for machine learning and data science. In this work, we introduce persistent Stanley–Reisner theory to bridge commutative algebra, combinatorial algebraic topology, machine learning, and data science. We propose persistent h-vectors, persistent f-vectors, persistent graded Betti numbers, persistent facet ideals, and facet persistence modules. Stability analysis indicates that these algebraic invariants are stable against geometric perturbations. We employ a machine learning prediction on a molecular dataset to demonstrate the utility of the proposed persistent Stanley–Reisner theory for practical applications.
 		\end{abstract}
	

\section{Introduction}
	
Topological Data Analysis (TDA) provides a mathematical framework for extracting topological and geometric features in high-dimensional, high-order, and complex data, with the ability to characterize directional, temporal, and functional properties. A central tool in TDA is persistent (co)homology, an algebraic topology technique that tracks how connected components, loops, and higher-dimensional voids are formed across a parametrized family of spaces \cite{edelsbrunner2008persistent, zomorodian2004computing, zomorodian2005computing}. Filtration-based methods (e.g., the scale parameter in Vietoris-Rips complexes) identify these features and encode their evolution in persistence modules, whose persistence barcodes or persistence diagrams capture the lifespan of each feature \cite{chazal2016structure,ghrist2008barcodes}.  
Subsequent variants, including persistent images \cite{adams2017persistence}, persistent landscapes \cite{persistenceLandscapes2015Bubenik}, and persistent Betti numbers, incorporate additional geometric and higher-dimensional information. These methodologies are empowered by 
machine learning  \cite{cang2015topological,henselSurveyTopologicalMachine2021}. For instance, in 2017, Cang and Wei first introduced the topological deep learning (TDL) paradigm \cite{cang2017topologynet}  by integrating persistent homology with deep neural networks.
TDL is the new frontier for relational learning \cite{papamarkou2024position}. Topological machine learning has found broad applications in protein folding \cite{xia2014persistent}, protein-ligand binding \cite{cang2017topologynet}, virus mutation \cite{chen2020mutations}, drug discovery \cite{nguyen2019mathematical}, computational chemistry \cite{townsend2020representation}, dynamical systems \cite{pereaTopologicalTimeSeries2019}, signal processing \cite{myersTeaspoonComprehensivePython2020},  and neural spike decoding \cite{mitchell2024topological}. 

Despite the success, persistent homology has notable limitations. For example, it cannot distinguish a five-member ring from a six-member ring, an essential distinction in molecular science, nor capture non-topological changes in network evolution. To address these challenges, two persistent topological Laplacian approaches, i.e., Persistent Spectral Graph (PSG) for point cloud data \cite{wang2020persistent} and evolutionary de Rham-Hodge method for data on manifolds \cite{chen2021evolutionary} were introduced in 2019. Persistent  topological Laplacians  provide both harmonic and non-harmonic data spectra \cite{wang2020persistent,memoli2022persistent,liu2024algebraic}. The harmonic spectra of PSGs are the same as the topological invariants of persistent homology, while non-harmonic spectra can capture additional homotopic shape evolution. Many persistent topological Laplacians have been proposed on different topological domains, including simplicial complex \cite{wang2020persistent}, path complex   \cite{wang2023persistent}, cellular sheaf \cite{wei2025persistent2}, directed flag complex,  hypergraphs \cite{liu2021persistent}, hyperdigraphs,  etc. Persistent Mayer topology extends the standard chain complex to the $N$-chain complex \cite{shen2024persistent,suwayyid2024persistentMayer}. Quantum topological algorithms through persistent Dirac operators have also been proposed \cite{ameneyro2024quantum, wee2023persistent,suwayyid2024persistentDirac}. 
TDA has also been extended to data on differentiable manifolds by using differential topology \cite{chen2021evolutionary, wee2021forman} and on one-dimensional (1D) curves embedded in 3-space by using geometric topology \cite{shen2024evolutionary}. The reader is referred to a survey for these advancements \cite{wei2025persistent}.  

Combinatorial commutative algebra is a relatively new mathematical field that integrates combinatorial topology and commutative algebra \cite{miller2005combinatorial,eisenbud2013commutative}, offering unique approaches for analyzing simplicial complexes. One of the major subjects in this field is the Stanley-Reisner ring, which encodes simplicial structure as a square-free monomial ideal in a polynomial ring \cite{stanley1996combinatorics, bruns1998cohen}. Hochster’s formula connects the ring’s graded Betti numbers to topological invariants of induced subcomplexes, linking discrete geometry, commutative algebra, combinatorics, and algebraic geometry.

Few methods, however, integrate the multiscale view of TDA with the refined algebraic invariants arising from the Stanley-Reisner theory. In this work, we introduce a persistent Stanley-Reisner theory (PSRT) and examine how the Stanley-Reisner structure of a simplicial complex evolves under filtration. We develop persistent analogs of classical invariants, including persistent graded Betti numbers via Hochster’s formula, persistent \( f \)-vectors, persistent \( h \)-vectors, and persistent facet ideals. We also define facet persistence barcodes, which record the birth and death of persistent facet ideals as the simplicial complex evolves. These persistence barcodes exhibit stability properties similar to those of standard persistent homology and provide novel insights into geometric, topological, and combinatorial features at multiple scales, complementing operator-based approaches such as persistent Laplacian and persistent Dirac methods.

We demonstrate the effectiveness of PSRT by capturing subtle multiscale changes in simplicial complexes. This technique can be applied to molecular structures and the classification of metal halide perovskite phases based on structural data. In many cases, persistent facet barcodes provide computationally and conceptually distinct perspectives, efficiently reflecting geometric and combinatorial substructures.  

The remainder of this paper is organized as follows. Section~2 reviews the classical Stanley-Reisner theory, its Hilbert series, and Hochster’s formula. Section~3 introduces the persistent Stanley-Reisner theory, including persistent graded Betti numbers, persistent h-vectors, f-vectors, persistent facet ideals, and  facet persistence module. Furthermore, Section~3 establishes stability theorems analogous to those in persistent homology, ensuring that a small perturbation of the filtration induces proportionally small changes in the facet persistence barcodes and critical values. Section~4 presents applications to molecular data, highlighting the strengths of these persistent commutative algebra in distinguishing subtle structural isomers and classification application.
	
\section{Stanley-Reisner Rings and Their Hilbert Series}\label{Stanley-ReisnerRing}
One of the central constructions in algebraic combinatorics is the 
\emph{Stanley-Reisner ring} (also called the \emph{face ring}) associated with 
a simplicial complex \cite{stanley1996combinatorics, bruns1998cohen}.  This construction encodes the combinatorial data of the complex into a graded commutative ring, thereby allowing one to apply tools 
from commutative algebra to study invariants such as homology and other 
combinatorial information.  In this section, we recall the basic definitions 
of the Stanley-Reisner ring and examine its Hilbert series, functions, and polynomials. 
Both the Hilbert series and Hilbert function are used in commutative algebra and algebraic geometry to study the growth of graded structures like polynomial rings and their quotients.  
In particular, we highlight the relationship between the $f$-vector of a 
simplicial complex and its $h$-vector, where the latter is reflected in the 
Hilbert function of the associated Stanley-Reisner ring.

\medskip
\noindent
Let \(\Delta\) be a \emph{simplicial complex} on the vertex set
\begin{equation}\label{eq:vertex-set}
	V \;=\; \{\,x_1, x_2, \dots, x_n\,\}.
\end{equation}
By definition, \(\Delta\) is a collection of subsets (called \emph{faces} or 
\emph{simplices}) of \(V\) satisfying:
\begin{enumerate}
	\item (Hereditary property) If \(F\in\Delta\) and \(G\subseteq F\), then 
	\(G\in\Delta\).
	\item Every singleton \(\{x_i\}\) belongs to \(\Delta\).  
	In particular, all vertices are included as faces.
\end{enumerate}
A face with \(r+1\) vertices is called an \emph{\(r\)-dimensional face}, 
and the \emph{dimension} of \(\Delta\) is the maximum dimension of its faces.  
A \emph{facet} of \(\Delta\) is a face that is maximal under inclusion, and 
we denote the set of all facets by \(\mathcal{F}(\Delta)\).

Let \(k\) be a field and consider the polynomial ring
\begin{equation}\label{eq:S-def}
	S \;=\; k[x_1, x_2, \dots, x_n],
\end{equation}
equipped with the standard \(\mathbb{Z}\)-grading defined by 
\(\deg(x_i) = 1\) for all \(i\).  The \emph{Stanley--Reisner ideal} of 
\(\Delta\) is given by
\begin{equation}\label{eq:SR-ideal}
	I(\Delta)
	\;=\;
	\Bigl\langle 
	x_{i_1} \,x_{i_2} \,\cdots\, x_{i_r}
	\;:\;
	\{\,x_{i_1},x_{i_2},\dots,x_{i_r}\}\notin\Delta 
	\Bigr\rangle,
\end{equation}
and the corresponding \emph{Stanley--Reisner ring} is the quotient
\begin{equation}\label{eq:SR-ring}
	k[\Delta]
	\;=\;
	S \,/\, I(\Delta),
\end{equation}
which inherits the \(\mathbb{Z}\)-grading from \(S\).  For each \(d\ge 0\), 
the homogeneous component of degree \(d\) is
\begin{equation}\label{eq:homogeneous-comp}
	k[\Delta]_d 
	\;=\;
	\bigl\{\overline{x}^\alpha : \lvert \alpha \rvert = d\bigr\},
\end{equation}
where \(\overline{x}^\alpha = \prod_{i=1}^n \overline{x}_i^{\,\alpha_i}\) and
\(\lvert \alpha \rvert = \alpha_1 + \cdots + \alpha_n\).

It is well known that
\begin{equation}\label{eq:Krull-dim}
	\dim \bigl(k[\Delta]\bigr)
	\;=\;
	\dim(\Delta) \;+\; 1.
\end{equation}

A finitely generated \(\mathbb{Z}\)-graded \(S\)-module 
\begin{equation}\label{eq:graded-module}
	M 
	\;=\;
	\bigoplus_{d\ge 0} M_d
\end{equation}
has a \emph{Hilbert function} defined by
\begin{equation}\label{eq:Hilbert-function}
	H(M,d) 
	\;=\;
	\dim_k (M_d).
\end{equation}
In the case \(M = k[\Delta]\), we set
\begin{equation}\label{eq:Hilbert-function-Delta}
	H(\Delta,d)
	\;=\;
	\dim_k\bigl(k[\Delta]_d\bigr),
\end{equation}
and define the \emph{Hilbert series} of \(\Delta\) by
\begin{equation}\label{eq:Hilbert-series-def}
	H_{\Delta}(s)
	\;=\;
	\sum_{d\ge 0} \dim_k \bigl(k[\Delta]_d\bigr)\,s^d
	\;=\;
	\sum_{d\ge 0} H(\Delta,d)\,s^d.
\end{equation}
For a \((d-1)\)-dimensional simplicial complex \(\Delta\), it is classical 
that
\begin{equation}\label{eq:Hilbert-series-classical}
	H_{\Delta}(s)
	\;=\;
	\frac{\,h_0 + h_1\,s + \cdots + h_d\,s^d\,}{(1-s)^d},
\end{equation}
where \(d = \dim(\Delta) + 1\) and \((h_0,h_1,\dots,h_d)\) is the 
\emph{\(h\)-vector} of \(\Delta\) (equivalently, of \(k[\Delta]\)).

As a graded \(S\)-module, \(k[\Delta]\) admits a minimal free resolution of 
the form
\begin{equation}\label{eq:min-free-res}
	\cdots 
	\;\longrightarrow\;
	\bigoplus_{j} S(-j)^{\beta_{i,j}\!\bigl(k[\Delta]\bigr)}
	\;\longrightarrow\;
	\cdots 
	\;\longrightarrow\;
	\bigoplus_{j} S(-j)^{\beta_{0,j}\!\bigl(k[\Delta]\bigr)}
	\;\longrightarrow\;
	k[\Delta]
	\;\longrightarrow\;
	0,
\end{equation}
where $ S(-j) $ is the graded free module $S$ shifted in degree by $j$ and  {\it graded Betti numbers} are
\begin{equation}\label{eq:Betti-numbers-def}
	\beta_{i,j}\!\bigl(k[\Delta]\bigr)
	\;=\;
	\dim_k \operatorname{Tor}^S_i\!\bigl(k[\Delta], k\bigr)_j,
\end{equation}
with ${\rm Tor}^S_i\!\bigl(k[\Delta], k\bigr)_j$ being the Tor module, which measures how nontrivial the resolution is at homological degree $i$.    
For a subset \(W \subseteq V\), the \emph{restriction} (or induced subcomplex) 
of \(\Delta\) to \(W\) is 
\begin{equation}\label{eq:induced-subcomplex}
	\Delta_W
	\;=\;
	\{\,
	\tau \in \Delta 
	: 
	\tau \subseteq W
	\}.
\end{equation}

A fundamental result in this theory is \emph{Hochster's formula}, which 
expresses the graded Betti numbers of \(k[\Delta]\) in terms of the reduced 
homology of its induced subcomplexes.

\begin{theorem}[Hochster's Formula]
	\label{thm:Hochster}
	For a simplicial complex \(\Delta\) on \(\{x_1,\dots,x_n\}\) and all integers 
	\(i,j \geq 0\),
	\begin{equation}\label{eq:Hochster-formula}
		\beta_{i,j}\!\bigl(k[\Delta]\bigr)
		\;=\;
		\dim_k \operatorname{Tor}^S_i\!\bigl(k[\Delta],k\bigr)_j 
		\;=\;
		\sum_{\substack{W \subseteq \{\,x_1,\dots,x_n\,\} \\ |W|=j}}
		\dim_k \widetilde{H}_{\,j-i-1}\!\bigl(\Delta_W; k\bigr),
	\end{equation}
	where \(\widetilde{H}_{r}\!\bigl(\Delta_W; k\bigr)\) is the 
	\(r\)th reduced homology group of \(\Delta_W\) with coefficients in \(k\).
\end{theorem}

\noindent
Hochster's formula shows that the homological properties of the induced 
subcomplexes of \(\Delta\) completely govern the graded Betti numbers of 
\(k[\Delta]\).  In particular, since \(\widetilde{H}_{j-i-1}(\Delta_W;k)\) 
is typically trivial for \(j-i-1 < -1\), it follows that 
\(\beta_{i,j}(k[\Delta])=0\) for \(j < i\).  Additional relationships include 
\(\beta_{i,i}(k[\Delta])=0\) for \(i \ge 1\), \(\beta_{0,0}(k[\Delta])=1\), 
and \(\beta_{i,1}(k[\Delta]) = \beta_{0,j}(k[\Delta])=0\) for \(j\ge 1\) and 
\(i\ge 0\).  Moreover, since \(\widetilde{H}_s(\Delta;k) = 0\) for 
\(s > \dim(\Delta)\), the summation in 
\eqref{eq:Hochster-formula} is typically nontrivial only for 
\(j \le \min\{n,\dim(\Delta)+i+1\}\).

Often, one restates Hochster's formula in the form
\begin{equation}\label{eq:Hochster-alt}
	\beta_{i,j+i}\!\bigl(k[\Delta]\bigr)
	\;=\;
	\sum_{\substack{W \subseteq \{\,x_1,\dots,x_n\,\} \\ |W|=j+i}}
	\dim_k \widetilde{H}_{\,j-1}\!\bigl(\Delta_W;k\bigr),
\end{equation}
for \(1 \le i \le n-1\) and \(1 \le j \le \min\{\,n-i,\dim(\Delta)+1\}\).  
In particular, for \(j=1\),
\begin{equation}\label{eq:Hochster-j1}
	\beta_{i,i+1}\!\bigl(k[\Delta]\bigr)
	\;=\;
	\sum_{\substack{W \subseteq \{\,x_1,\dots,x_n\,\} \\ |W| = i+1}}
	\bigl(\,\beta_{0}(\Delta_W) - 1\bigr),
\end{equation}
and for \(j\ge 2\),
\begin{equation}\label{eq:Hochster-j2}
	\beta_{i,j+i}\!\bigl(k[\Delta]\bigr)
	\;=\;
	\sum_{\substack{W \subseteq \{\,x_1,\dots,x_n\,\} \\ |W| = j+i}}
	\beta_{j-1}(\Delta_W),
\end{equation}
where $\beta_{0}$ and $\beta_{j-1}$ are the Betti number of the corresponding reduced homology groups.
These formulas explicitly connect the combinatorial topological invariants of \(\Delta\) 
to the algebraic structure of its Stanley--Reisner ring. 
While one can arrange these Betti numbers in a table with rows indexed by \(j\) and columns by \(i\), 
\[
\begin{array}{c|cccc}
	& i=0 & i=1 & i=2 & \cdots \\ \hline
	j=0 & \beta_{0,0} & \beta_{1,0} & \beta_{2,0} & \cdots \\
	j=1 & \beta_{0,1} & \beta_{1,1} & \beta_{2,1} & \cdots \\
	j=2 & \beta_{0,2} & \beta_{1,2} & \beta_{2,2} & \cdots \\
	\vdots & \vdots & \vdots & \vdots & \ddots
\end{array}
\]
However, due to an alternative convention, the indexing in the Betti table differs from the above. The graded Betti numbers are written as \( \beta_{i,j} \), where \( i \) denotes the homological degree and \( j \) the total internal degree. However, in many Betti tables (such as those produced by the Macaulay2 package), the table entries are displayed at position \( (i, j-i) \), where the row index corresponds to the difference \( j-i \).
\[
\begin{array}{c|cccc}
	& i = 0 & i = 1 & i = 2 & \cdots \\ \hline
	j - i = 0 & \beta_{0,0} & \beta_{1,1} & \beta_{2,2} & \cdots \\
	j - i = 1 & \beta_{0,1} & \beta_{1,2} & \beta_{2,3} & \cdots \\
	j - i = 2 & \beta_{0,2} & \beta_{1,3} & \beta_{2,4} & \cdots \\
	\vdots & \vdots & \vdots & \vdots & \ddots
\end{array}
\]
	
	\begin{example}
		
		\begin{figure}[h!]
			\centering
			\subfigure{
				\includegraphics[width=0.33\textwidth]{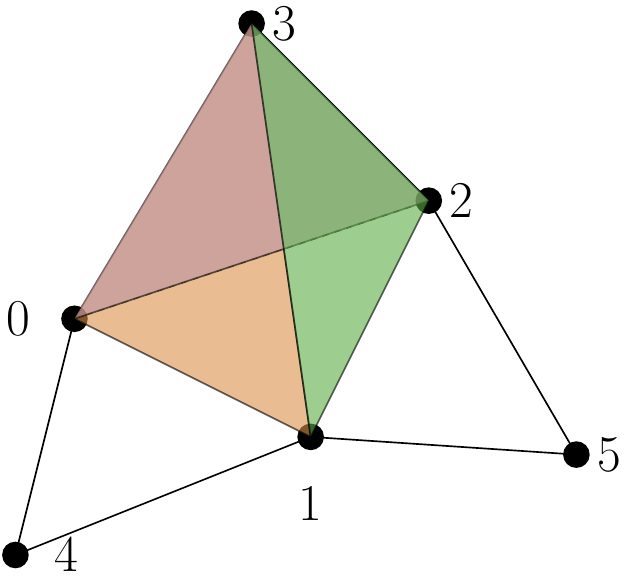} 
			}
			\caption{A geometric representation of six vertices forming a pyramid with two distinct loops, each attached to different edges.}
			\label{fig:pyramid}
		\end{figure}
		Graded Betti numbers provide insight into the structure and complexity of a module, particularly in algebraic geometry and commutative algebra. We consider the simplicial complex depicted in Figure \ref{fig:pyramid}. We compute its graded Betti table $ \beta_{i,j}$ via Hochster's formula. The faces of the simplicial complex are:
		\[
		\begin{aligned}
			&\{x_0\}, \quad \{x_1\}, \quad \{x_2\}, \quad \{x_3\}, \quad \{x_4\}, \quad \{x_5\}, \\
			&\{x_0, x_1\}, \quad \{x_0, x_2\}, \quad \{x_0, x_3\}, \quad \{x_0, x_4\}, \quad \{x_1, x_2\}, \\
			&\{x_1, x_3\}, \quad \{x_1, x_4\}, \quad \{x_1, x_5\}, \quad \{x_2, x_3\}, \quad \{x_2, x_5\} \\
			&\{x_0, x_1, x_2\},\quad \{x_0, x_1, x_3\},\quad \{x_0, x_2, x_3\}, \quad \{x_1, x_2, x_3\}.
		\end{aligned}
		\]
		Observe that the entries of the first column of the table are identically zero except for the first one, which is equal to one; this is because, among any \(j\) vertices, there are insufficient \((j-1)\)-dimensional faces to contribute nontrivially to the kernel of the corresponding boundary operator in the associated chain complex. Consequently, our computation commences with the row corresponding to \(j=2\), where only \(\beta_{1,2}\) requires evaluation. In this case, \(\beta_{1,2}\) is determined by enumerating all pairs of vertices not connected by an edge, yielding exactly five such pairs.

		\begin{table}[h]
			\centering
    \begin{tabular}{c|ccccc}
	& 0 & 1 & 2 & 3 & 4 \\ 
	\midrule
	0  & 1  & 0  & 0  & 0  & 0  \\
	1  & 0  & 5  & 0  & 0  & 0  \\
	2  & 0  & 2  & 6  & 0  & 0  \\
	3  & 0  & 1  & 6  & 2  & 0  \\
	4  & 0  & 0  & 2  & 6  & 0  \\
	5  & 0  & 0  & 0  & 1  & 2  \\
\end{tabular}
			\caption{Table of graded Betti numbers of the simplicial complex in Figure \ref{fig:pyramid}.}
			\label{tab:betti}
		\end{table}
		
Proceeding to the case \(j=3\), we evaluate both \(\beta_{1,3}\) and \(\beta_{2,3}\). Here, \(\beta_{2,3}\) is defined as the number of three vertex subsets containing at least one isolated vertex, which is found to be six. In contrast, \(\beta_{1,3}\) is obtained by counting all the triangles (3-cycles) present in the simplicial complex, which results in precisely two.

For \(j=4\), we determine the values \(\beta_{i,4}\) for \(i=1,2,3\). Specifically, \(\beta_{3,4}\) is computed by identifying all four-vertex subsets whose induced subcomplex is disconnected (i.e., consists of at least two connected components), resulting in two such subsets. The computation of \(\beta_{2,4}\) involves enumerating the four-vertex subsets whose induced subcomplex contains a cycle (a three-vertex loop); this yields six subsets, each featuring exactly one loop. Finally, \(\beta_{1,4}\) is derived by identifying those four-vertex subsets whose induced subcomplex exhibits at least one void, leading to a unique instance corresponding to the pyramid.

This iterative procedure is continued until the entire table of graded Betti numbers is wholly determined. The full result is shown in Table \ref{tab:betti}.

	\end{example}

	\section{ Persistent Stanley-Reisner Theory}
	In this section, we introduce \emph{persistent Stanley-Reisner theory}, including \emph{persistent Stanley-Reisner Graded Betti numbers} obtained as extensions of Hochster’s Formula over filtration, \emph{persistent h-vectors}, \emph{persistent f-vectors}, the \emph{persistent facet ideals}, and \emph{facet persistence modules}. We also analyze the \emph{stability} of the persistent Stanley-Reisner theory. We begin by defining a class of functions that naturally align with the combinatorial structure of abstract simplicial complexes.
	
		Let \(\Delta\) be an abstract simplicial complex on a vertex set \(V\), and let \(f: \Delta \to \mathbb{R}\) be a real-valued function defined on the simplices of \(\Delta\). The function \(f\) is said to be \emph{monotonic} if for every pair of simplices \(\sigma, \tau \in \Delta\) with \(\tau \subseteq \sigma\), we have
		\[
		f(\tau) \leq f(\sigma).
		\]
		In other words, \(f\) is monotonic if the value of \(f\) increases (or remains constant) as we move from a face to a containing simplex.
		Given a monotonic function \(f\colon \Delta \to \mathbb{R}\), we denote its induced filtration by
		\begin{equation}\label{eq:filtration}
			\widetilde{f} \;=\; (\Delta_f^t)_{t\in \mathbb{R}},
		\end{equation}
		where each \(\widetilde{f}(t) = \Delta_f^t \subseteq \Delta\) is a subcomplex defined by
		\[
		\Delta_f^t \;=\; \{\sigma \in \Delta \mid f(\sigma) \le t\}.
		\]
		Because \(f\) is \emph{non-decreasing} on simplices, whenever \(t \le r\) we have 
		\[
		\Delta^t_f \;\subseteq\; \Delta^r_f.
		\]
		We often write \(\Delta^t\) instead of \(\widetilde{f}(t)\) or \(\Delta^t_f\) if there is no ambiguity.  The persistent Stanley--Reisner ideal of each subcomplex \(\Delta^t\) is denoted \(I(\Delta^t)\).
	Throughout this section, let \(\Delta\) be a simplicial complex, and let
	\(f \colon \Delta \to \mathbb{R}\) be a monotonic function.

	\subsection{Persisetnt Stanley-Reisner Graded Betti Numbers and Extensions of Hochster's Formula}
	
	In the context of a family of simplicial complexes indexed by a real parameter, 
persistent homology provides a powerful way to track how topological features 
(e.g., connected components, cycles) evolve as the parameter changes.  
We begin by recalling the definition of persistent homology groups in the single-parameter setting, followed by an induced subcomplex filtration and its connection to Hochster's formula.

\begin{definition}[Persistent Homology Groups]\label{def:persistent-homology}
	Let \((\Delta^t)_{t \in \mathbb{R}}\) be a filtration of simplicial complexes 
	with \(\Delta^t \subseteq \Delta^{t'}\) whenever \(t \le t'\).  For each \(q \ge 0\) 
	and each pair of real numbers \(t \le t'\), the inclusion 
	\(\Delta^t \subseteq \Delta^{t'}\) induces a linear map on homology
	\begin{equation}\label{eq:inclusion-map}
		\iota_{q}^{t,t'}
		\;:\;
		\widetilde{H}_q(\Delta^t; k)
		\;\longrightarrow\;
		\widetilde{H}_q(\Delta^{t'}; k).
	\end{equation}
	The \emph{persistent homology group} in degree \(q\) from \(t\) to \(t'\) is 
	defined as 
	\begin{equation}\label{eq:persistent-homology-group}
		\mathrm{Im}\bigl(\iota_q^{t,t'}\bigr) 
		\;=\;
		\iota_{q}^{t,t'}\bigl(\,\widetilde{H}_q(\Delta^t; k)\bigr)
		\;\subseteq\;
		\widetilde{H}_q(\Delta^{t'}; k).
	\end{equation}
	Equivalently, \(\mathrm{Im}(\iota_q^{t,t'})\) measures which homology classes 
	in \(\Delta^t\) remain nontrivial in \(\Delta^{t'}\).  In this sense, 
	\(\mathrm{Im}(\iota_q^{t,t'})\) quantifies the ``lifespan'' of topological 
	features in degree \(q\) across the filtration.
\end{definition}

Persistent homology groups form the foundation for constructing barcodes 
or persistence diagrams, visually representing topological features' birth and death.  These invariants play a central role in 
topological data analysis.

Let \(\Delta\) be a fixed simplicial complex on a vertex set \(V\).  
For any subset \(W \subseteq V\), the \emph{restriction} (or induced subcomplex) 
of \(\Delta\) to \(W\) is denoted \(\Delta_W\).  Suppose 
\((\Delta^t)_{t \in \mathbb{R}}\) is a filtration of \(\Delta\).  Then the 
\emph{induced subcomplex filtration} on \(W\) is given by
\begin{equation}\label{eq:induced-subcomplex-filtration}
	\Delta_W^t
	\;=\;
	\Delta^t \,\cap\, \Delta_W 
	\;\subseteq\;
	\Delta_W^{t'}
	\;=\;
	\Delta^{t'} \,\cap\, \Delta_W 
	\quad \text{whenever} \quad t \le t'.
\end{equation}
Consequently, for \(t \le t'\), there is an induced linear map on homology
\begin{equation}\label{eq:induced-inclusion-map}
	\iota_{q}^{t,t'} 
	\;:\;
	\widetilde{H}_q(\Delta_W^t; k)
	\;\longrightarrow\;
	\widetilde{H}_q(\Delta_W^{t'}; k).
\end{equation}
This allows one to examine persistent homology restricted to subsets \(W\) 
of the vertex set \(V\). Recall that in the non-persistent setting, Hochster's formula expresses the graded Betti numbers \(\beta_{i,j}(k[\Delta])\) of the Stanley--Reisner ring \(k[\Delta]\) in terms of the ranks of the homology groups 
\(\widetilde{H}_{j-i-1}(\Delta_W;k)\).  To adapt this relationship to a 
filtration \((\Delta^t)_{t \in \mathbb{R}}\), one replaces the 
(non-persistent) homology ranks by \emph{persistent homology ranks} of the 
maps
\(\iota_{j-1}^{t,t'} : \widetilde{H}_{j-1}(\Delta_W^t; k) \rightarrow 
\widetilde{H}_{j-1}(\Delta_W^{t'}; k)\).

\begin{definition}[\index{Persistent Stanley--Reisner Graded Betti Numbers}Persistent Stanley--Reisner Graded Betti Numbers]\label{def:persistent-Betti-numbers}
	Let \(\Delta\) be a simplicial complex on the vertex set 
	\(\{x_1,\dots,x_n\}\).  Suppose \((\Delta^t)_{t \in \mathbb{R}}\) is a 
	filtration of \(\Delta\).  For each pair \((t,t')\) with \(t \le t'\), and 
	integers \(i,j \ge 0\), define the persistent Stanley--Reisner graded Betti number to be
	\begin{equation}\label{eq:persistent-Betti}
		\beta_{i, i+j}^{t,t'}\bigl(k[\Delta]\bigr) 
		\;=\;
		\sum_{\substack{W \subseteq \{x_1, \dots, x_n\} \\ |W| = i + j}}
		\dim_k \Bigl(
		\iota_{j-1}^{t,t'} 
		\;:\; 
		\widetilde{H}_{\,j-1}(\Delta_W^t; k) 
		\;\longrightarrow\; 
		\widetilde{H}_{\,j-1}(\Delta_W^{t'}; k)
		\Bigr).
	\end{equation}
	Here, \(\Delta_W^t\) is the induced subcomplex of \(\Delta^t\) restricted to 
	\(W\), and
	\(\iota_{j-1}^{t,t'}\) is the inclusion-induced map in degree \((j-1)\) 
	from \(\Delta_W^t\) to \(\Delta_W^{t'}\).
\end{definition}

The quantity
\[
\dim_{k}\bigl(\iota_{j-1}^{t,t'} : \widetilde{H}_{j-1}(\Delta_W^t; k) 
\;\to\; \widetilde{H}_{j-1}(\Delta_W^{t'}; k)\bigr)
\]
represents the rank of the persistent homology group from \(\Delta_W^t\) to \(\Delta_W^{t'}\).  
The sum in \eqref{eq:persistent-Betti} accumulates these ranks over all subsets \(W\) of the appropriate cardinality, providing an algebraic topological measure of how many \((j-1)\)-dimensional features persist from scale \(t\) to \(t' \).  
Therefore, the persistent Stanley--Reisner graded Betti numbers  \(\beta_{i,i+j}^{t,t'}\bigl(k[\Delta]\bigr)\) in \eqref{eq:persistent-Betti} record the number of homological features of ``dimension" \((j-1)\) (captured by subsets \(W\) of cardinality \(i+j\)) that persist from \(t\) to \(t' \).  This framework extends Hochster's formula to a multiscale setting, permitting barcodes or diagrams that encode the intervals \([t,t']\) over which prime (facet) structures or induced homology groups remain present. 

Persistent Stanley-Reisner graded Betti numbers also include the standard persistent Betti numbers (for instance, \(\beta_{i,|V|}^{t,t'} = \beta_{|V|-i-1}^{t,t'}\)) and incorporate additional invariants by accounting for every level of the simplicial complex.

\subsection{Persistent h-vectors and Persistent f-vectors}

In many geometric and combinatorial settings (for instance, in the study of projective varieties, toric ideals, and Stanley--Reisner rings), knowledge of the graded Betti numbers and the Hilbert polynomial yields detailed information about the structure of a module, including invariants such as the projective dimension and Castelnuovo--Mumford regularity \cite{stanley1996combinatorics, bruns1998cohen}. In particular, the following lemma establishes a connection between the $h$-vector and the graded Betti numbers.
We recall a helpful fact about finitely generated graded modules 
over a polynomial ring, which connects their graded Betti numbers to their 
Hilbert series.

\begin{lemma}[Graded Free Resolution and Hilbert Series]\label{lemma:graded-free-res-Hilb}
	Let \(S\) be a polynomial ring in \(n\) variables over a field \(k\), 
	endowed with the standard \(\mathbb{Z}\)-grading, and let 
	\(\dim(S) = n\).  Suppose \(M\) is a finite graded \(S\)-module of finite 
	projective dimension. Then, there exists a graded free resolution of \(M\)
	\begin{equation}\label{eq:graded-free-res}
		0 
		\;\longrightarrow\;
		\bigoplus_{j} S(-j)^{\beta_{p,j}}
		\;\longrightarrow\;
		\cdots
		\;\longrightarrow\;
		\bigoplus_{j} S(-j)^{\beta_{0,j}}
		\;\longrightarrow\;
		M 
		\;\longrightarrow\;
		0,
	\end{equation}
	whose graded Betti numbers are \(\beta_{i,j}\).  The Hilbert series of \(M\), 
	denoted \(H_M(t)\), then satisfies
	\begin{equation}\label{eq:Hilbert-series-general}
		H_M(s)
		\;=\;
		\frac{Q_M(s)}{\,(1-s)^n},
		\quad\text{where}\quad
		Q_M(s) 
		\;=\;
		\sum_{i,j} (-1)^i\,\beta_{i,j}\,s^j.
	\end{equation}
\end{lemma}

When \(M = k[\Delta]\) is the Stanley-Reisner ring of a \((d-1)\)-dimensional 
simplicial complex \(\Delta\), one often writes 
\(\dim\bigl(k[\Delta]\bigr) = d\).  In this case, we have the following 
relation between the polynomial \(Q_{k[\Delta]}(t)\) and the classical 
\(h\)-vector \(\bigl(h_0, h_1, \dots, h_d\bigr)\).  First, observe that
\begin{equation}\label{eq:SR-Hilbert-series}
	(1 - s)^{\,n - d}
	\Bigl(\,\sum_{m=0}^{d} h_m\, s^m \Bigr)
	\;=\;
	\sum_{i,j}\,(-1)^i\,\beta_{i,j}\,s^j,
\end{equation}
where \(n = \dim(S)\).  
Setting
\begin{equation}\label{eq:Bj-def}
	B_j 
	\;:=\; 
	\sum_{i=0}^{p} (-1)^i\, \beta_{i,j},
	\quad j=0,1,\dots,n,
\end{equation}
where $p+1$ is the  length of a minimal free resolution, we may rewrite the polynomial \(Q_{k[\Delta]}(t)\) from 
\eqref{eq:Hilbert-series-general} as
\begin{equation}\label{eq:QkDelta}
	Q_{k[\Delta]}(s)
	\;=\;
	\sum_{j=0}^{\,n} B_j\, s^j.
\end{equation}
Hence, using \eqref{eq:SR-Hilbert-series}, one obtains
\begin{equation}\label{eq:Ht-formula}
	Q_{k[\Delta]}(s)
	\;=\;
	(1 - s)^{\,n-d}\, H(s),
	\quad \text{where} \quad 
	H(s) = \sum_{m=0}^{\,d} h_m\, s^m.
\end{equation}
Inverting this,
\begin{equation}\label{eq:Ht-inversion}
	H(s)
	\;=\;
	(1-s)^{-(n-d)}\,Q_{k[\Delta]}(s).
\end{equation}
Since
\begin{equation}\label{eq:binomial-expansion}
	(1 - s)^{-(n-d)}
	\;=\;
	\sum_{k \,\ge\, 0}
	\binom{\,n-d + k - 1\,}{\,k\,}\, s^k,
\end{equation}
we can extract the coefficients \((h_m)\) of \(H(s)\) by collecting 
appropriate terms of \(Q_{k[\Delta]}(s)\).  In particular,
\begin{equation}\label{eq:h-in-terms-of-B}
	h_m
	\;=\;
	\sum_{j=0}^{\,m} 
	\binom{\,n-d + m - j - 1\,}{\,m-j\,}
	B_j
	\;=\;
	\sum_{j=0}^{\,m}
	\binom{\,n-d + m - j - 1\,}{\,m-j\,}
	\Bigl(\,
	\sum_{i=0}^{\,p}
	(-1)^i \,\beta_{i,j}
	\Bigr).
\end{equation}
Thus, the \(h\)-vector \(\bigl(h_0,\dots,h_d\bigr)\) is expressed directly 
in terms of the alternating sums \(\sum_{i=0}^{p} (-1)^i \beta_{i,j}\).

For a \((d-1)\)-dimensional simplicial complex \(\Delta\), the \emph{$f$-vector} 
is given by
\begin{equation}\label{eq:f-vector-def}
	(f_0, f_1, \dots, f_{d-1}),
\end{equation}
where \(f_i\) denotes the number of \(i\)-dimensional faces of \(\Delta\).  
By convention, we also set \(f_{-1} = 1\) to account for the empty face.  
The $f$-vector and $h$-vector are classically related by
\begin{equation}\label{eq:f-in-terms-of-h}
	\sum_{j=0}^{\,d} h_j \, s^j
	\;=\;
	\sum_{j=0}^{\,d} f_{\,j-1}\,\bigl(1 - s\bigr)^{\,d - j}\, s^j,
	\quad
	\text{with}
	\quad
	f_{-1} = 1.
\end{equation}
Equivalently, one obtains
\begin{equation}\label{eq:f-h-binomial}
	f_{\,j-1}
	\;=\;
	\sum_{i=0}^j
	\binom{\,d - i\,}{\,j - i\,}\; h_i,
	\quad
	j = 0,1,\dots,d.
\end{equation}
Conversely, the $h$-vector can be recovered from the $f$-vector via
\begin{equation}\label{eq:h-in-terms-of-f}
	h_j
	\;=\;
	\sum_{i=0}^{\,j}
	(-1)^{\,j - i}\,\binom{\,d - i\,}{\,j - i\,}\,
	f_{\,i-1},
	\quad
	j = 0,1,\dots,d.
\end{equation}

Combining the binomial relationship \eqref{eq:f-h-binomial} with the 
expression \eqref{eq:h-in-terms-of-B} for each \(h_m\) in terms of the 
graded Betti numbers \(\beta_{\,i,j}\), one obtains a direct formula for 
each \(f_{\,k-1}\) in terms of these Betti numbers:
\begin{equation}\label{eq:f-vector-graded-betti}
	f_{\,k-1}
	\;=\;
	\sum_{i=0}^k 
	\binom{\,d - i\,}{\,k - i\,}
	\sum_{j=0}^{\,i}
	\binom{\,n-d + i - j - 1\,}{\,i - j\,}
	\biggl(\,
	\sum_{r=0}^p 
	(-1)^r \,\beta_{\,r,j}
	\biggr),
	\quad k=0,1,\dots,d.
\end{equation}
Equivalently, one may swap the order of summation in \(i\) and \(j\) to write
\begin{equation}\label{eq:f-vector-graded-betti-alt}
	f_{\,k-1}
	\;=\;
	\sum_{j=0}^k 
	\biggl(\,
	\sum_{r=0}^p 
	(-1)^r \,\beta_{\,r,j}
	\biggr)\,
	\sum_{i=j}^k
	\binom{\,d - i\,}{\,k - i\,}
	\binom{\,n-d + i - j - 1\,}{\,i - j\,}.
\end{equation}
Both \eqref{eq:f-vector-graded-betti} and \eqref{eq:f-vector-graded-betti-alt} 
express the entries of the $f$-vector purely in terms of the graded Betti 
numbers \(\beta_{\,r,j}\).  In this sense, one obtains a complete algebraic 
description of the combinatorial $f$-vector (and hence the $h$-vector) via 
the data encoded in a minimal free resolution of \(k[\Delta]\).

Using \eqref{eq:h-in-terms-of-B} and \eqref{eq:f-h-binomial}, one may extend the definitions of \(f\)-vectors and \(h\)-vectors to multiscale and persistent versions. Specifically, we define the \index{Persistent $h$-vector}persistent $h$-vector to be
\begin{equation}\label{persistence_h}
	h_m^{t,t'}
	\;=\;
	\sum_{j=0}^{m}
	\binom{\,n-d + m - j - 1\,}{\,m-j\,}
	\Bigl(\,
	\sum_{i=0}^{p}
	(-1)^i \,\beta_{i,j}^{t,t'}
	\Bigr),
\end{equation}
and the \index{Persistent $f$-vector}persistent $f$-vector to be
\begin{equation}\label{persistence_f}
	f_{\,m-1}^{t,t'}
	\;=\;
	\sum_{i=0}^m
	\binom{\,d - i\,}{\,m - i\,}\; h_i^{t,t'},
	\quad
	m = 0,1,\dots,d.
\end{equation}

	\begin{example}

	\begin{figure}[h!]
	\centering
	\subfigure[]{
		\includegraphics[width=0.25\textwidth]{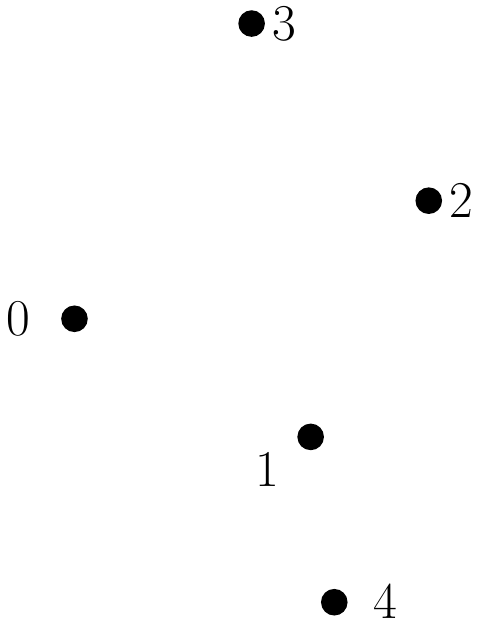} 
		\label{fig:pyramid1}
	}
	\subfigure[]{
		\includegraphics[width=0.25\textwidth]{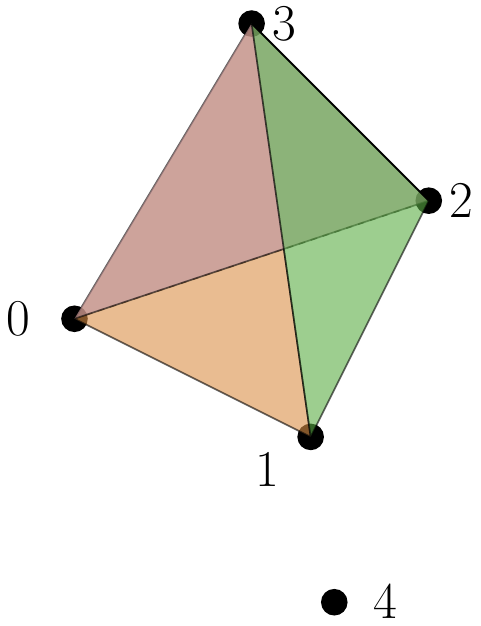} 
		\label{fig:pyramid2}
	}
	\subfigure[]{
		\includegraphics[width=0.25\textwidth]{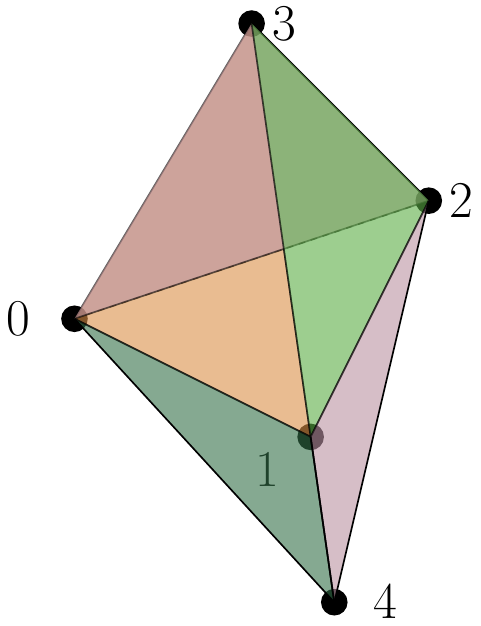} 
		\label{fig:pyramid3}
	}
	\caption{Three representative plots of the filtration of a triangular bipyramid with an equatorial triangular cross-section.
	}
	\label{fig:pyramid_filtration}
\end{figure}

In this example, we analyze the triangular bipyramid shown in Figure~\ref{fig:pyramid_filtration}, focusing on its Betti numbers and graded Betti numbers, as well as the corresponding $f$-vectors, $h$-vectors, and their persistent variants. Specifically, we compare $\beta_{0}-1$ (number of connected components minus one), $\beta_{1}$ (loops), and $\beta_{2}$ (voids) to the graded Betti numbers $\beta_{4,5}$ (number of connected components minus one), $\beta_{1,3}$ (loops at the level of three vertices), $\beta_{2,4}$ (loops at the level of four vertices), and $\beta_{2,5}$ (voids). The faces of the complete simplicial complex are:
\[
\begin{aligned}
	&\{x_0\}, \quad \{x_1\}, \quad \{x_2\}, \quad \{x_3\}, \quad \{x_4\}, \\
	&\{x_0, x_1\}, \quad \{x_0, x_2\}, \quad \{x_0, x_3\}, \quad \{x_0, x_4\}, \quad \{x_1, x_2\}, \\
	&\{x_1, x_3\}, \quad \{x_1, x_4\}, \quad \{x_2, x_3\}, \quad \{x_2, x_4\}, \\
	&\{x_0, x_1, x_3\}, \quad \{x_0, x_1, x_4\}, \quad \{x_0, x_2, x_3\}, \quad \{x_0, x_2, x_4\}, \quad \{x_1, x_2, x_3\}, \quad \{x_1, x_2, x_4\}.
\end{aligned}
\]
Unlike the standard Betti numbers, the graded Betti numbers capture loops at the three- and four-vertex levels. Figure~\ref{fig:persistence_plot} illustrates that a loop appears and remains persistent at the three-vertex level, indicating a triangular configuration. At the four-vertex level, a loop forms but does not persist for an extended duration, indicating that the loop primarily lives at the three-vertex level. The figure also displays selected components of the $f$-vector and $h$-vector curves, highlighting their persistent counterparts. Table~\ref{tab:betti2} presents the graded Betti numbers at the critical parameter values $0$, $1$, and $2$.

\begin{table}[h!]
	\centering
	\begin{minipage}{0.32\textwidth}
		\centering
		\begin{tabular}{c|ccccc}
			& 0 & 1 & 2 & 3 & 4 \\ 
			\midrule
			0  & 1  & 0  & 0  & 0  & 0  \\
			1  & 0  & 10 & 0  & 0  & 0  \\
			2  & 0  & 0  & 20 & 0  & 0  \\
			3  & 0  & 0  & 0  & 15 & 0  \\
			4  & 0  & 0  & 0  & 0  & 4  \\
		\end{tabular}
		\label{tab:first}
	\end{minipage}%
	\begin{minipage}{0.32\textwidth}
		\centering
		\begin{tabular}{c|ccccc}
			& 0 & 1 & 2 & 3 & 4 \\ 
			\midrule
			0  & 1  & 0  & 0  & 0  & 0  \\
			1  & 0  & 4  & 0  & 0  & 0  \\
			2  & 0  & 1  & 6  & 0  & 0  \\
			3  & 0  & 0  & 1  & 4  & 0  \\
			4  & 0  & 0  & 0  & 0  & 1  \\
		\end{tabular}
		\label{tab:second}
	\end{minipage}%
	\begin{minipage}{0.32\textwidth}
		\centering
		\begin{tabular}{c|ccccc}
			& 0 & 1 & 2 & 3 & 4 \\ 
			\midrule
			0  & 1  & 0  & 0  & 0  & 0  \\
			1  & 0  & 1  & 0  & 0  & 0  \\
			2  & 0  & 1  & 0  & 0  & 0  \\
			3  & 0  & 0  & 1  & 0  & 0  \\
		\end{tabular}
		\label{tab:third}
	\end{minipage}
	\caption{From left to right: The betti tables at the critical values of the filtration in Figure \ref{fig:pyramid_filtration}.}
	\label{tab:betti2}
\end{table}

\begin{figure}[h!]
	\centering
	\subfigure{
		\includegraphics[width=0.9\textwidth]{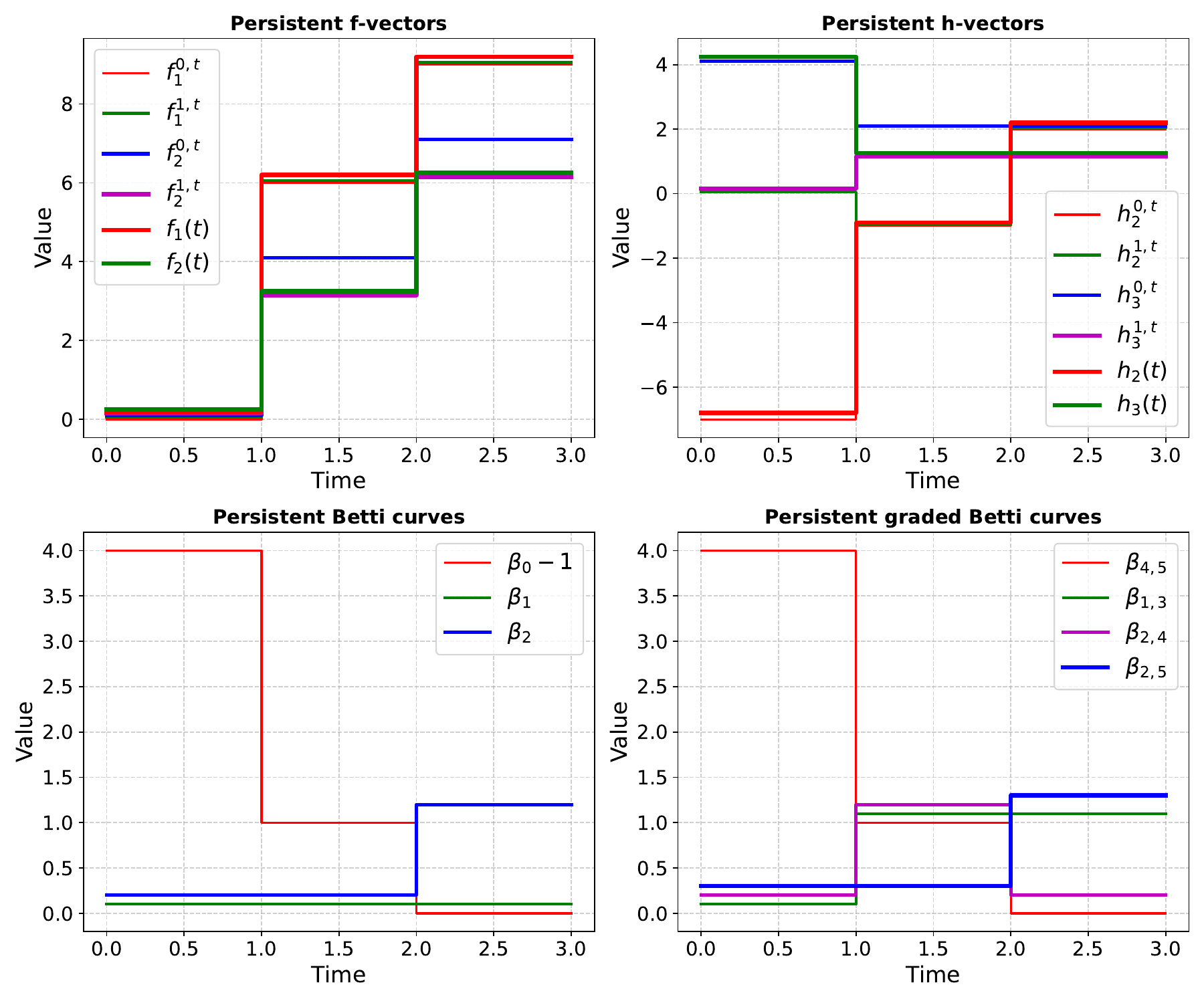} 
		\label{fig:persistence}
	}
	\caption{An illustration of the persistent variations of the \( h \)-vectors, \( f \)-vectors, Betti numbers, and graded Betti numbers, respectively. Each curve represents a function that assumes discrete integer values. To enhance the visual clarity of the plots and prevent overlap, a slight increment is added to the curves, ensuring that each curve remains distinguishable within the graphical representation.
	}
	\label{fig:persistence_plot}
\end{figure}
\end{example}

\subsection{Persistent Facet Ideals}

Because Stanley-Reisner ideals are generated by squarefree monomials, they are radical and therefore admit a primary decomposition as an intersection of their minimal prime ideals~\cite{atiyah2018introduction}. These prime ideals reflect combinatorial properties of the associated simplicial complex, as elaborated upon in the following discussion. Before proceeding, we present the following definitions.

\begin{definition}[Prime Monomial Ideals]
Let \(V = \{x_1, x_2, \ldots, x_n\}\) be a finite set of indeterminates over a field \(k\). For each subset \(A \subseteq V\), define the \emph{prime monomial ideal} 
\[
P_A \;:=\; \bigl(x_i \,\mid\, x_i \not \in  A\bigr).
\]
Equivalently, for any prime monomial ideal \(P \subseteq k[V]\), set
\[
W \;:=\; \{\, x_i \in V \mid x_i \notin P \}.
\]
Then \(P = P_W\). In particular, there is a bijection
\[
\{\text{subsets } A \subseteq V\} 
\;\;\longleftrightarrow\;\;
\{\text{prime monomial ideals } P \subseteq k[V]\}
\]
given by \(A \mapsto P_A\) and \(P \mapsto \{\,x_i \in V : x_i \notin P\}\).
\end{definition}

Let $\Delta$ be a simplicial complex. In Stanley–Reisner theory, we focus on prime monomial ideals associated with the facets of $\Delta$, which we shall refer to as facet prime monomial ideals, or simply facet ideals, for reasons that will become clear below.

Let \(I(\Delta)\) be the \emph{Stanley--Reisner ideal} of the simplicial complex \(\Delta\).  By definition, \(I(\Delta)\) is generated by all monomials corresponding to the non-faces of \(\Delta\). It is a radical ideal and factors as the intersection of its minimal prime ideals:
\[
I(\Delta) \;=\; \bigcap_{i} P_i.
\]
Each minimal prime \(P_i\) in this factorization corresponds precisely to a unique facet of \(\Delta\). Precisely,
\[
I(\Delta) \;=\; \bigcap_{\sigma\in \mathcal{F}(\Delta)} P_{\sigma},
\]
where $\mathcal{F}(\Delta)$ denotes the collection of facets of $\Delta$. Now let \((\Delta^t)_{t \in \mathbb{R}}\) be a filtration of \(\Delta\).  As \(t\) varies, the factorization of \(I(\Delta^t)\) into its facet prime monomial ideals evolves in accordance with the reversed inclusion
\[
I(\Delta^{t'}) \;\subseteq\; I(\Delta^t).
\]
We denote by \(\mathcal{P}(\Delta^t)\) the collection of all facet ideals of \(\Delta^t\). For each \(i \geq 0\), let \(\mathcal{P}_i(\Delta^t)\) be the subcollection of \(\mathcal{P}(\Delta^t)\) consisting of the facet ideals corresponding to \(i\)-dimensional facets of \(\Delta^t\). Then, we have the disjoint union
\[
\mathcal{P}(\Delta^t) = \bigsqcup_{i=0}^{\dim(\Delta^t)} \mathcal{P}_i(\Delta^t).
\]
Motivated by this observation, we introduce a notion analogous to the birth and vanishing of topological features from classical persistence theory. Therefore, we call the facet prime monomial ideals of \(I(\Delta^t)\) as the \index{Persistent facet ideals}\emph{persistent facet ideals}.
Before describing this construction, we first present the necessary definitions.
\begin{definition}[\index{Stanley--Reisner Critical Value}Stanley--Reisner Critical Value]\label{def:SR-critical-value}
A \emph{Stanley--Reisner critical value} of $f$ is a real number \(r\) such that, for a sufficiently small \(\varepsilon > 0\),
\begin{equation}
	I\bigl(\Delta_f^{\,r+\varepsilon}\bigr) \;\subsetneq\; I\bigl(\Delta_f^{\,r-\varepsilon}\bigr),
	\label{eq:strict-inclusion}
\end{equation}
i.e., there is a strict inclusion of the corresponding Stanley--Reisner ideals. 
\end{definition}

These critical values indicate where the Stanley-Reisner ideal and its decomposition change, and they also enable a planar representation of \(f\). We now present a typical lemma in this context.

\begin{lemma}[Critical Value Lemma]\label{lem:critical-value}
	If a closed interval \([x,y]\) does not contain a critical value of the filtration \((\Delta^t)_{t \in \mathbb{R}}\), then 
	\begin{equation}
		I\bigl(\Delta^y\bigr) \;=\; I\bigl(\Delta^x\bigr).
		\label{eq:critical-lemma}
	\end{equation}
\end{lemma}

\begin{proof}
	The case $x = y$ is trivial, so assume $x < y$. Set $J = [x,y]$ and note that $J$ contains no critical values by assumption.  
	Hence, there exists \(\epsilon > 0\) such that \(x + \epsilon < y\) and
	\begin{equation}
		I\bigl(\Delta^x\bigr) \;=\; I\bigl(\Delta^r\bigr)
		\quad \text{for all } x \le r \le x + \epsilon.
		\label{eq:interval-equality}
	\end{equation}
	Define
	\begin{equation}
		\alpha \;=\; \sup\Bigl\{\,r \,\Big\vert\, I\bigl(\Delta^x\bigr) = I\bigl(\Delta^r\bigr),\, x \le r \le y\Bigr\}.
		\label{eq:sup-alpha}
	\end{equation}
	Since $x < \alpha$ and $\alpha$ is not a critical value, it follows (by the definition of $\alpha$) that 
	\begin{equation}
		I\bigl(\Delta^x\bigr) \;=\; I\bigl(\Delta^\alpha\bigr).
		\label{eq:equal-at-alpha}
	\end{equation}
	We claim that $\alpha = y$. Indeed, if $\alpha < y$, then by choice of $\alpha$ (and the fact that $\alpha$ is not a critical value), we obtain a contradiction to its definition as a supremum where \(I(\Delta^x) = I(\Delta^r)\). Hence, \(\alpha = y\), and the result in Eq. \eqref{eq:critical-lemma} follows.
\end{proof}

We now introduce the birth and death indices to track the emergence and disappearance of prime monomial ideals in filtration.
\begin{definition}[Birth, Death, and Lifespan of Persistent Facet Ideals]
	\label{def:birth-death}
	Let \(\sigma \in \Delta\) be a face, and let \(P_{\sigma}\) be the corresponding prime monomial ideal. Consider the filtration \((\Delta^t)_{t\in\mathbb{R}}\). 
	\begin{enumerate}
		\item The \emph{birth} of \(P_{\sigma}\) is the smallest real number \(b\) such that \(P_{\sigma}\) is a minimal prime over \(I(\Delta^b)\) (equivalently, corresponds to a facet of \(\Delta^b\)), and for every sufficiently small \(\varepsilon > 0\), \(P_{\sigma}\) is not a minimal prime over \(I(\Delta^{\,b-\varepsilon})\).  
		If no finite \(b\) satisfies these conditions, we say the birth is \emph{at infinity} and set \(b = \infty\). We denote this birth index by \(b_P\).
		
		\item The \emph{death} of \(P_{\sigma}\) is the largest real number \(d\) such that \(P_{\sigma}\) is \emph{not} a minimal prime over \(I(\Delta^d)\) but becomes minimal over \(I(\Delta^{\,d+\varepsilon})\) for all sufficiently small \(\varepsilon > 0\).  
		If no finite \(d\) satisfies these conditions, we say the death is \emph{at infinity} and set \(d = \infty\). We denote this death index by \(d_P\).
		
		\item The \emph{lifespan} of \(P_{\sigma}\) is the (possibly unbounded) interval
		\[
		[b_P,\,d_P),
		\]
		where \(b_P\) and \(d_P\) are the birth and death indices of \(P_{\sigma}\) respectively.
	\end{enumerate}
\end{definition}

\noindent
Thus, much like in persistent homology, we obtain a \emph{barcode} representation for the persistent facet ideals: each facet prime \(P_{\sigma}\) appears (is born) at some parameter value \(b\) and disappears (dies) at some parameter value \(d\).  

Critical values correspond to points where the decomposition of the persistent Stanley-Reisner ideal changes. Each critical value is associated with the appearance or disappearance of a persistent facet ideal in the primary decomposition of the persistent Stanley-Reisner ideal. Conversely, every change, e.g., the birth and death indices, occurs at a critical value.
Let $t \le t'$ and define the \index{Facet persistent Betti number}\emph{facet persistent Betti number} \(\beta_t^{t'}\) to be the number of minimal primes over \(I(\Delta^t)\) that remain minimal over \(I(\Delta^{t'})\).  
Let \(\alpha_1 < \alpha_2 < \cdots < \alpha_n\) be the Stanley--Reisner critical values of \((\Delta^t)_{t \in \mathbb{R}}\).  
Set \(\alpha_{0} = -\infty\) and \(\alpha_{n+1} = +\infty\).  
Consider an interleaved sequence \(\bigl(b_i\bigr)_{i=-1,0,1,\dots,n}\) such that \(b_{i-1} < \alpha_i < b_i\) for all \(i\), with \(b_{-1} = -\infty\) and \(b_{n+1} = +\infty\).  
For two integers \(i,j\) with \(0 \le i < j \le n+1\), we define the \emph{multiplicity} of the pair \((\alpha_i, \alpha_j)\) by
\begin{equation}
	\mu_i^j 
	\;=\; 
	\beta^{b_j}_{b_{\,i-1}} 
	\;-\; 
	\beta^{b_j}_{b_i} 
	\;+\; 
	\beta^{b_{\,j-1}}_{b_i} 
	\;-\; 
	\beta^{b_{\,j-1}}_{b_{\,i-1}}.
	\label{eq:mu-definition}
\end{equation}
\begin{definition}
The \index{Facet persistence diagram}\emph{facet persistence diagram} of the Stanley--Reisner critical values, denoted
\[
D(f) = D\bigl((\Delta^t)_{t \in \mathbb{R}}\bigr) \;\subset\; \overline{\mathbb{R}}^2,
\]
is then the set of points \((\alpha_i,\alpha_j)\), counted with multiplicity \(\mu_i^j\) for \(0 \le i < j \le n+1\), together with all points on the diagonal, each counted with infinite multiplicity. 
\end{definition}
We aim to prove that for two monotonic functions \(f,g\colon \Delta \to \mathbb{R}\),
\begin{equation}
	d_b\Bigl(D\bigl((\Delta_f^t)\bigr),\,D\bigl((\Delta_g^t)\bigr)\Bigr)
	\;\le\; 
	\|f - g\|_{\infty},
	\label{eq:bottleneck-inequality}
\end{equation}
where \(d_b\) is the \emph{bottleneck distance} (defined in later parts of the article). Observe that the multiplicity \(\mu_i^j\) can also be written as
\begin{equation}
	\mu_i^j
	\;=\;
	\Bigl(
	\beta^{b_{\,j-1}}_{\,b_i} 
	\;-\; 
	\beta^{b_{\,j-1}}_{\,b_{\,i-1}}
	\Bigr)
	\;-\;
	\Bigl(
	\beta^{b_j}_{\,b_i} 
	\;-\; 
	\beta^{b_j}_{\,b_{\,i-1}}
	\Bigr).
	\label{eq:mu-rewritten}
\end{equation}
In this formulation, the difference
\[
\beta_{\,b_{i}}^{\,b_{j-1}}
\;-\;
\beta_{\,b_{i-1}}^{\,b_{j-1}}
\]
precisely counts the minimal primes over 
\(\displaystyle I\bigl(\Delta^{b_{\,i}}\bigr)\)
that were \emph{not} minimal over 
\(\displaystyle I\bigl(\Delta^{b_{\,i-1}}\bigr)\)
and that remain minimal up to 
\(\displaystyle I\bigl(\Delta^{b_{\,j-1}}\bigr)\).
Since the sequence \((\,b_i)\) is interleaved so that
\(\,b_{\,i-1} < \alpha_i < b_{\,i},\)
and the \(\alpha_i\) are precisely the critical values, 
these newly ``born'' minimal primes do not appear at any parameter value smaller than \(b_{\, i}\). Hence, the quantity 
\(\mu_i^j\)
measures the number of minimal primes that emerge by \(b_{\,i}\) and then cease to be minimal by \(b_{\,j-1}\).  
This observation motivates the subsequent definitions and the stability analysis for the bottleneck distance.

\begin{definition}[\index{Facet Persistence Barcodes}Facet Persistence Barcodes]
	Let \(\widetilde{f} = (\Delta^t)_{t \in \mathbb{R}}\) be a filtration. The \emph{barcode} \(\mathcal{B}(f)\), or simply \(\mathcal{B}\), associated to \(\widetilde{f}\) is the multiset of horizontal line segments
	\[
	[b_P,\, d_P)
	\]
	in the plane, each corresponding to the birth--death interval of a prime monomial ideal that becomes minimal over some \(I(\Delta^t)\) in the filtration. $\mathcal{B}$ is called \emph{Facet Persistence Barcodes} of $\widetilde{f}$.
	
	Let
	\[
	\overline{\mathcal{H}} \;=\; \{(p,q)\,\mid\, -\infty \leq p \leq q \leq +\infty\}
	\]
	denote the \emph{extended half-plane}.
	The \emph{diagram} of the filtration, denoted
	\[
	\mathrm{dgm}(f) \;=\; \mathrm{dgm}(\widetilde{f}) 
	\;=\; \mathrm{dgm}\bigl((\Delta^t)_{t \in \mathbb{R}}\bigr),
	\]
	is the set of points
	\[
	\{(b_{\sigma}, d_{\sigma})\}\,\cup\,\{(p,p)\,\mid\, p \in \mathbb{R}\}
	\;\subset\; \overline{\mathcal{H}},
	\]
	often drawn together with the line segments \([b_{\sigma}, d_{\sigma}]\) in the extended plane.
\end{definition}
Notice that then,
\[
\mathrm{dgm}(f) = D(f).
\]

	\subsubsection{Facet Persistence Module and Stability}

Given a filtration \(\widetilde{f} = (\Delta^t)_{t \in \mathbb{R}}\) and the induced family of ideals \((I(\Delta^t))_{t \in \mathbb{R}}\), we define a persistence module as follows.

\begin{definition}[\index{Facet Persistence Module}Facet Persistence Module]
	\label{def:SR-prime-PM}
	Let \(\Delta\) be a simplicial complex, and for each \(t \in \mathbb{R}\), let $\mathcal{P}(\Delta^t)$ be the collection of the persistent facet ideals of the persistent Stanley--Reisner ideal \(I(\Delta^t)\). Define the vector space
	\[
	V_t \;=\; \bigoplus_{P \in \mathcal{P}(\Delta^t)} k,
	\]
	where each direct summand \(k\) is spanned by a basis vector corresponding to \(P \in \mathcal{P}(\Delta^t)\).
	
	For each pair \((s,t)\) with \(s \le t\), the linear map
	\[
	v^{\,s}_{\,t} \;:\; V_s \;\longrightarrow\; V_t
	\]
	is defined by sending a basis vector corresponding to \(P \in \mathcal{P}(\Delta^s)\) to the same basis vector in \(V_t\) if \(P\) is in $P \in \mathcal{P}(\Delta^t)$, and to \(0\) otherwise.
	
	In other words, $V_t$ is simply spanned by the minimal primes over the persistent Stanley--Reisner ideal $I(\Delta^t)$, and the linear map \(v^{\,s}_{\,t}\)
	is defined by sending a basis vector corresponding to a minimal prime$P$ over $I(\Delta^s)$ to itself if \(P\) remains a minimal prime of \(I(\Delta^t)\), and to \(0\) otherwise.
	
	The resulting collection of vector spaces \((V_t)_{t \in \mathbb{R}}\) and linear maps \((v^{\,s}_{\,t})_{s \le t}\) forms a persistence module, denoted by
	\[
	\mathbb{V}^f 
	\;=\; 
	\Bigl( (V_t)_{t \in \mathbb{R}}, \bigl(v^{\,s}_{\,t}\bigr)_{s \le t} \Bigr),
	\]
	which we call the \emph{facet persistence module} of $\widetilde{f}$.
\end{definition}

By construction, notice that $\beta_t^{t'}= \text{rank}(v_t^{t'})$. Since each \(V_t\) is a finite-dimensional \(k\)-vector space for every \(t \in \mathbb{R}\), the persistence module \(\mathbb{V}^f\) is \(q\)-tame.  Consequently, by the structure theorem for \(q\)-tame persistence modules, the persistence module can be decomposed as a direct sum of interval modules \cite{chazal2016structure}. We state the explicit form of this decomposition below. First, we recall the definition of the interval modules as defined in \cite{chazal2016structure}.

\begin{definition}[Interval Persistence Module]
	Let \(I\) be an interval of \(\mathbb{R}\). The \emph{interval persistence module} \(\mathbf{k}^I\) over \(\mathbb{R}\) is
	defined by 
	\[
	(\mathbf{k}^I)_t \;=\; \begin{cases}
		k & \text{if } t\in I,\\
		0 & \text{otherwise}.
	\end{cases}
	\]
	The structure maps \(\bigl(\mathbf{k}^I\bigr)_t \to \bigl(\mathbf{k}^I\bigr)_{t'}\) are the identity on \(k\) whenever \(t,\, t' \in I\), and zero otherwise.
\end{definition}

When the interval \( I \) is given explicitly—for example, \( I = [a, b) \)—we simplify notation by writing \( \mathbf{k}^I = \mathbf{k}[a, b) \). The main result and the corresponding explicit decomposition are presented below.

\begin{theorem}\cite{chazal2016structure}[Decomposition of Facet Persistence Modules]
	\label{thm:SR-module-decomp}
	Using the notation from Definition~\ref{def:SR-prime-PM}, we have
	\[
	\mathbb{V}^f 
	\;\cong\; 
	\bigoplus_{[b_{P},\,d_{P}) \,\in\, \mathcal{B}}
	\mathbf{k}[b_{P},\, d_{P}),
	\]
	where \([b_{P},\, d_{P})\) denotes the birth--death interval of the persistent facet ideal \(P\).  In particular, each  persistent facet ideal \(P\) contributes a single interval module \(k[b_{P},\, d_{P})\), and
	\[
	P \in \mathcal{P}(\Delta^t)
	\quad\Longleftrightarrow\quad
	t \in [b_{P},\, d_{P}).
	\]
\end{theorem}

\begin{proof}[Sketch of Proof]
	By definition, a prime \(P\) lies in \(\mathcal{P}(\Delta^t)\) if and only if it is minimal over \(I(\Delta^t)\).  The birth--death indices \(b_{P}\) and \(d_{P}\) capture exactly the range of \(t\)-values for which \(P\) remains minimal over \(I(\Delta^t)\).  In other words, \(P\) is active in the module \(V_t\) precisely for \(t \in [b_{P}, d_{P})\).  Standard arguments in the classification of persistence modules over totally ordered sets (see, e.g., interval decompositions in \cite{chazal2016structure}) then imply the stated direct sum decomposition.
\end{proof}

\noindent
Hence, one obtains 
\(\mathrm{dgm}\bigl(\mathbb{V}^f\bigr) =  \mathrm{dgm}\bigl(f\bigr)\).
Let \(\mathbb{U} = ( (U_t)_{t \in \mathbb{R}}, (uchrom^{\,s}_{\,t})_{s \le t} )\) and 
\(\mathbb{V} = ( (V_t)_{t \in \mathbb{R}}, (v^{\,s}_{\,t})_{s \le t} )\) 
be persistence modules over \(\mathbb{R}\). Recall that each \(U_t\) and \(V_t\) is a vector space (or module), and 
\[ u_s^t : U_s \longrightarrow U_t, \quad v_s^t : V_s \longrightarrow V_t \]
are linear maps (the \emph{structure maps}) defined for all \(s \le t\).

\begin{definition}\cite{chazal2016structure}[Homomorphism of Degree \(\delta\)]
	\label{def:degree-delta-hom}
	Let \(\delta \in \mathbb{R}\). A \emph{homomorphism of degree \(\delta\)} from \(\mathbb{U}\) to \(\mathbb{V}\) is a collection of linear maps 
	\[
	\varphi_t : U_t \;\longrightarrow\; V_{t+\delta}
	\]
	for every \(t \in \mathbb{R}\), such that for all \(s \le t\) the following diagram commutes:
	\[
	\begin{tikzcd}
		U_s \arrow[r, "u_s^t"] \arrow[d, "\varphi_s"'] & U_t \arrow[d, "\varphi_t"] \\
		V_{s+\delta} \arrow[r, "v_{s+\delta}^{t+\delta}"'] & V_{t+\delta}.
	\end{tikzcd}
	\]
	We write \(\Phi \in \mathrm{Hom}^\delta(\mathbb{U}, \mathbb{V})\) if \(\Phi = (\varphi_t)_{t\in\mathbb{R}}\) is such a degree-\(\delta\) homomorphism.
\end{definition}

\begin{definition}\cite{chazal2016structure}[\(\delta\)-Interleaving]
	\label{def:delta-interleaving}
	Let \(\delta \ge 0\) and consider two persistence modules \(\mathbb{U}\) and \(\mathbb{V}\). We say \(\mathbb{U}\) and \(\mathbb{V}\) are \(\delta\)-\emph{interleaved} if there exist homomorphisms
	\[
	\Phi \;=\; (\varphi_t)_{t\in\mathbb{R}} \;\in\; \mathrm{Hom}^\delta(\mathbb{U}, \mathbb{V}), 
	\quad
	\Psi \;=\; (\psi_t)_{t\in\mathbb{R}} \;\in\; \mathrm{Hom}^\delta(\mathbb{V}, \mathbb{U}),
	\]
	such that 
	\[
	\Psi \circ \Phi \;=\; 1_\mathbb{U}^{2\delta}, 
	\quad 
	\Phi \circ \Psi \;=\; 1_\mathbb{V}^{2\delta}.
	\]
	Concretely, this means there are maps
	\[
	\varphi_t : U_t \to V_{t+\delta} 
	\quad\text{and}\quad 
	\psi_t : V_t \to U_{t+\delta}
	\]
	for all \(t \in \mathbb{R}\), satisfying the appropriate commutative diagrams that ensure the structure of each persistence module is respected. In particular, one obtains:
	\[
	\begin{tikzcd}
		U_s \arrow[r, "u_s^t"] \arrow[d, "\varphi_s"'] & U_t \arrow[d, "\varphi_t"] \\
		V_{s+\delta} \arrow[r, "v_{s+\delta}^{t+\delta}"'] & V_{t+\delta}
	\end{tikzcd}
	\quad
\begin{tikzcd}[row sep=2.5em, column sep=1.5em]
	U_{s-\delta} \arrow[rr, "u_{s+\delta}^{s-\delta}"] \arrow[dr, "\varphi_{s-\delta}"'] & & U_{s+\delta} \\
	& V_s \arrow[ur, "\psi_s"']
\end{tikzcd}
	\]
	\[
	\begin{tikzcd}
		V_s \arrow[r, "v_s^t"] \arrow[d, "\psi_s"'] & V_t \arrow[d, "\psi_t"] \\
		U_{s+\delta} \arrow[r, "u_{s+\delta}^{t+\delta}"'] & U_{t+\delta}
	\end{tikzcd}
	\quad
\begin{tikzcd}[row sep=2.5em, column sep=1.5em]
	V_{s-\delta} \arrow[rr, "v_{s+\delta}^{s-\delta}"] \arrow[dr, "\psi_{s-\delta}"'] & & V_{s+\delta} \\
	& U_s \arrow[ur, "\varphi_s"']
\end{tikzcd}
	\]
	which guarantee \(\mathbb{U}\) and \(\mathbb{V}\) are interleaved by \(\delta\).
\end{definition}

The \index{Interleaving distance}\emph{interleaving distance} \(d_i(\mathbb{U}, \mathbb{V})\) between two persistence modules \(\mathbb{U}\) and \(\mathbb{V}\) is then defined by
\[
d_i(\mathbb{U}, \mathbb{V}) \;=\; \inf \{\delta \,\mid\, \mathbb{U}\text{ and }\mathbb{V}\text{ are }\delta\text{-interleaved}\}.
\]

We begin by introducing a distance on point sets in the extended half-plane, which will be used to define the bottleneck distance. Recall that 
\[
\mathcal{H} \;=\; \{(p,q) \mid p \leq q\},
\]
which captures all pairs \((p,q)\) satisfying \(p \leq q\). 

Next, we extend \(\mathcal{H}\) by allowing coordinates to take the values \(-\infty\) and \(+\infty\).  Specifically, the \emph{extended half-plane} \(\overline{\mathcal{H}}\) is decomposed into the disjoint union 
\[
\overline{\mathcal{H}} 
\;=\;
\mathcal{H}
\;\cup\;
\bigl(\{-\infty\}\times\mathbb{R}\bigr)
\;\cup\;
\bigl(\mathbb{R}\times\{+\infty\}\bigr)
\;\cup\;
\bigl\{(-\infty,+\infty)\bigr\}.
\]
Each component encodes a different boundary or interior region of the extended plane. 

We now define a distance \(d^\infty\) on \(\overline{\mathcal{H}}\). First, for any points \((p,q)\) and \((r,s)\) in \(\mathcal{H}\), set 
\[
d^\infty\bigl((p,q),(r,s)\bigr) 
\;=\; 
\max\{|p - r|, |q - s|\}.
\]
This measures the usual \(\ell^\infty\)-distance when both points lie strictly within \(\mathcal{H}\). 

We next consider points lying on the extended boundaries. For any \(q, s \in \mathbb{R}\), define
\[
d^\infty\bigl((-\infty,q),(-\infty,s)\bigr)
\;=\;
|q - s|,
\quad
d^\infty\bigl((q,+\infty),(s,+\infty)\bigr)
\;=\;
|q - s|,
\]
and
\[
d^\infty\bigl((-\infty,+\infty),(-\infty,+\infty)\bigr)
\;=\;
0.
\]
Finally, to ensure that points in different components of \(\overline{\mathcal{H}}\) remain far apart, we set 
\(
d^\infty(u,v) \;=\; +\infty 
\)
whenever 
\(
u,v
\)
lie in different subsets of the above decomposition.
This completes the definition of the distance \(d^\infty\), which establishes a natural way to measure proximity in the extended half-plane.

Let 
\(
\Delta(\mathbb{R}^2)
\;=\;
\{(p,p)\,\mid\,p\in\mathbb{R}\}
\)
be the diagonal in the plane. For a point \((p,q)\) with \(p < q\), the relevant distance from \((p,q)\) to the diagonal in the \(\ell^\infty\)-metric is commonly defined as
\[
d^\infty\bigl((p,q),\Delta\bigr)
\;=\;
\frac{1}{2}\,(q - p).
\]

Let \(A\) and \(B\) be two (multi)sets of points in the extended plane \(\overline{\mathcal{H}}\).

\begin{definition}\cite{chazal2016structure}[\index{Partial Matching}Partial Matching]
	A \emph{partial matching} between \(A\) and \(B\) is a subset 
	\[
	M \;\subseteq\; A \times B
	\]
	such that each \(\alpha \in A\) is paired with at most one \(\beta \in B\) and each \(\beta \in B\) is paired with at most one \(\alpha \in A\). Equivalently, \((\alpha,\beta), (\alpha,\beta')\in M\) implies \(\beta=\beta'\), and similarly for points of \(B\).
\end{definition}

	A partial matching \(M \subseteq A \times B\) is called a \(\delta\)\emph{-matching} if
	\begin{enumerate}
		\item For every \((\alpha, \beta) \in M\), \(d^\infty(\alpha, \beta) \le \delta\).
		\item If \(\alpha \in A\) is unmatched, then \(d^\infty(\alpha, \Delta) \le \delta\).
		\item If \(\beta \in B\) is unmatched, then \(d^\infty(\beta, \Delta) \le \delta\).
	\end{enumerate}

We are now prepared to formally recall the bottleneck distance.
\begin{definition}[\index{Bottleneck Distance}Bottleneck Distance]
	\label{def:bottleneck}
	The \emph{bottleneck distance} \(d_b(A, B)\) between two multisets \(A, B \subseteq \overline{H}\) is 
	\[
	d_b(A, B) 
	\;=\; 
	\inf \bigl\{\,\delta \,\mid\, \exists \text{ a }\delta\text{-matching between }A \text{ and } B \bigr\}.
	\]
\end{definition}

We now relate interleavings of persistence modules to matchings of their persistence diagrams. Recall \(\mathrm{dgm}(\mathbb{U})\) denotes the set (multiset) of points in the extended plane encoding the barcode or persistence diagram of \(\mathbb{U}\).

\begin{theorem}\cite{chazal2016structure}[Stability under Interleavings]
	\label{thm:stability}
	Let \(\mathbb{U}\) and \(\mathbb{V}\) be \(q\)-tame persistence modules that are \(\delta\)-interleaved.  Then there exists a \(\delta\)-matching between the multisets \(\mathrm{dgm}(\mathbb{U})\) and \(\mathrm{dgm}(\mathbb{V})\).  In particular,
	\[
	d_b\bigl(\mathrm{dgm}(\mathbb{U}),\,\mathrm{dgm}(\mathbb{V})\bigr) \;\le\; \delta.
	\]
\end{theorem}

Having recalled the essential definitions and results from \cite{chazal2016structure}, we now establish the main proof of our result.

\begin{theorem}[Stability of Facet Persistence Barcodes and Diagrams]
	\label{thm:prime-stability}
	Let \(f, g : \Delta \to \mathbb{R}\) be two real-valued functions on a simplicial complex \(\Delta\), and let \((\Delta^t_f)\) and \((\Delta^t_g)\) be the induced filtrations of \(\Delta\).  Denote by 
	\(\mathrm{dgm}(f)\) and \(\mathrm{dgm}(g)\) their corresponding Stanley-Reisner persistence diagrams.  Then
	\[
	d_b\bigl(\mathrm{dgm}(f),\,\mathrm{dgm}(g)\bigr)
	\;\le\;
	\|f - g\|_{\infty},
	\]
	where \(\|f - g\|_{\infty} = \sup_{\sigma \in \Delta}\,\bigl|\,f(\sigma) - g(\sigma)\bigr|\).
\end{theorem}

\begin{proof}
	Set \(\delta := \|f - g\|_{\infty}\).  By definition of the supremum norm, for every simplex \(\sigma \in \Delta\),
\[
|f(\sigma) - g(\sigma)| \;\le\; \delta,
\]
which implies
\[
f(\sigma) \;\le\; g(\sigma) + \delta 
\quad\text{and}\quad
g(\sigma) \;\le\; f(\sigma) + \delta.
\]
From \(f(\sigma) \le t\) we get \(g(\sigma) \le t+\delta\), yielding 
\[
\Delta^t_f 
\;=\; \{\sigma \in \Delta \mid f(\sigma)\le t\}
\;\subseteq\;
\{\sigma \in \Delta \mid g(\sigma)\le t+\delta\}
\;=\; \Delta^{t+\delta}_g.
\]
On the level of the persistent Stanley--Reisner ideals, this inclusion of simplicial complexes reverses, so
\[
I\bigl(\Delta^{\,t+\delta}_g\bigr)
\;\subseteq\;
I\bigl(\Delta^t_f\bigr).
\]
Recall that each facet persistence module is built via the minimal primes over these ideals (namely, the facet prime monomial ideals corresponding to facets). Since 
\[
P_\sigma \supseteq I\bigl(\Delta^t_f\bigr) 
\;\Longrightarrow\;
P_\sigma \supseteq I\bigl(\Delta^{\,t+\delta}_g\bigr),
\]
the birth index of a given face prime \(P_\sigma\) in the filtration for \(f\) can differ from its birth index in the filtration for \(g\) by at most \(\delta\).  Concretely,
\[
b_{P_\sigma}^f 
\;\le\; b_{P_\sigma}^g + \delta
\quad\text{and}\quad
b_{P_\sigma}^g 
\;\le\; b_{P_\sigma}^f + \delta,
\]
so \(\bigl|\,b_{P_\sigma}^f - b_{P_\sigma}^g\bigr| \le \delta\).  A similar argument shows that the corresponding death indices \(d_{P_\sigma}^f\) and \(d_{P_\sigma}^g\) also differ by at most \(\delta\). 

Consequently, one obtains a way to construct a \(\delta\)-interleaving between the persistence modules associated with the prime decompositions (facets) of \(\Delta^t_f\) and \(\Delta^t_g\). Specifically, for each \(t\in \mathbb{R}\) one defines linear maps
\[
\phi_t : V^f_t \;\longrightarrow\; V^g_{t+\delta}
\quad\text{and}\quad 
\psi_t : V^g_t \;\longrightarrow\; V^f_{t+\delta}
\]
by sending basis elements (coming from minimal primes over \(I(\Delta^t_f)\)) either to corresponding basis elements (if they remain minimal primes over \(I(\Delta^{t+\delta}_g)\)) or to zero otherwise.  One check that these maps form a degree-\(\delta\) homomorphism that respects the persistence structure, i.e., \ the standard diagrams
\[
\begin{tikzcd}
	V^f_s \arrow[r] \arrow[d, "\phi_s"'] & V^f_t \arrow[d, "\phi_t"] \\
	V^g_{s+\delta} \arrow[r] & V^g_{t+\delta}
\end{tikzcd}
\quad\text{and}\quad
\begin{tikzcd}
	V^g_s \arrow[r] \arrow[d, "\psi_s"'] & V^g_t \arrow[d, "\psi_t"] \\
	V^f_{s+\delta} \arrow[r] & V^f_{t+\delta}
\end{tikzcd}
\]
commute for \(s \le t\).  Thus \(\mathbb{V}^f\) and \(\mathbb{V}^g\) are \(\delta\)-interleaved.

By the standard stability result for interleaved persistence modules, e.g., Theorem \ref{thm:stability}, the bottleneck distance between \(\mathrm{dgm}(f)\) and \(\mathrm{dgm}(g)\) is at most \(\delta\). Therefore,
\[
d_b\bigl(\mathrm{dgm}(f),\,\mathrm{dgm}(g)\bigr) 
\;\le\; 
\|f - g\|_{\infty},
\]
as claimed.
\end{proof}

	\section{Application}
	
		\begin{figure}[h!]
		\centering
		\subfigure{
			\includegraphics[width=0.99\textwidth]{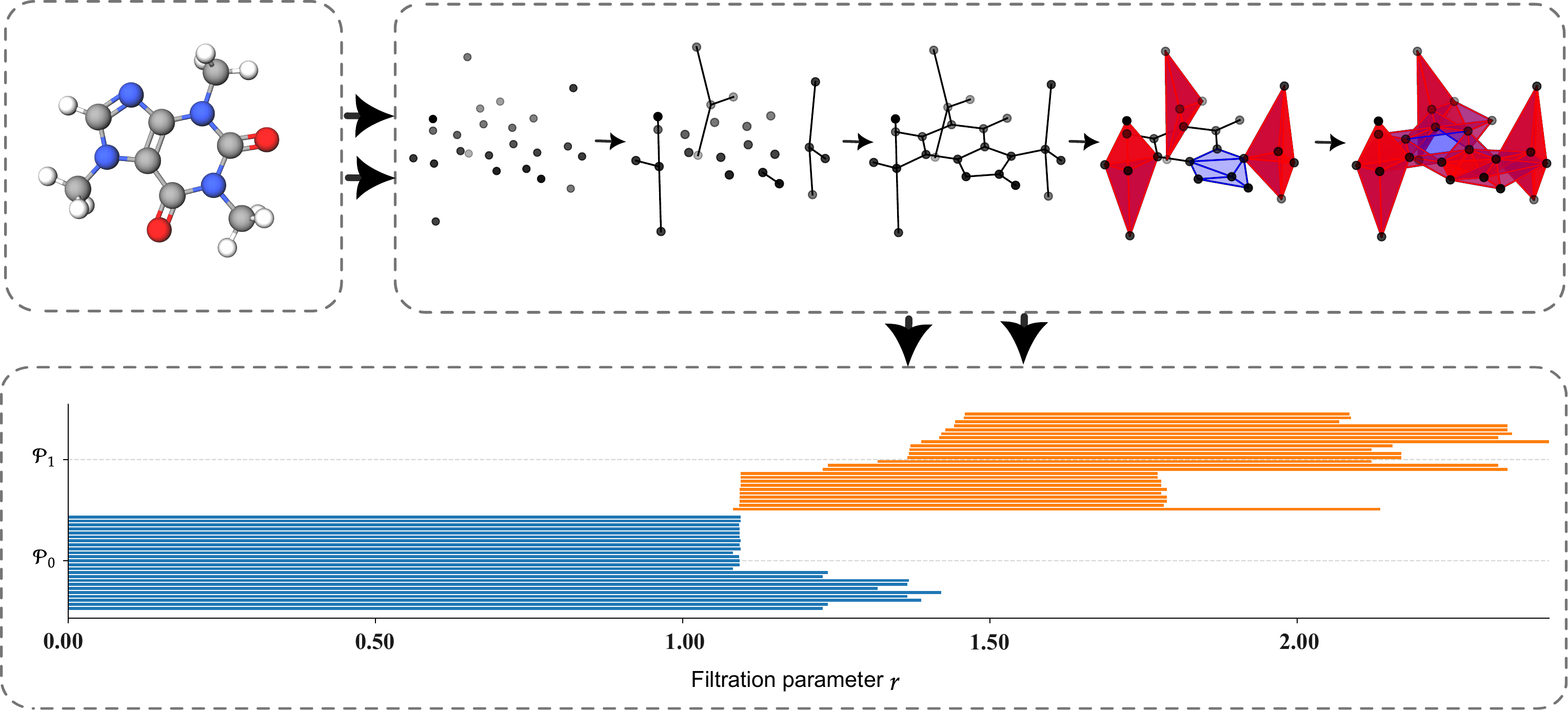} 
		}
		\caption{Illustration of the methodology and key steps in the proposed application of persistent Stanley--Reisner theory. Given a molecular input, a corresponding simplicial complex with an associated filtration is generated. Critical values and facet persistence barcodes are computed, leading to the construction of relevant features.
		}
		\label{fig:scheme}
	\end{figure}

Recent developments in artificial intelligence have prompted the design of models that encode molecular identity through composition and atomic arrangement. Approaches grounded in topological data analysis and persistent homology have been employed in various biological and molecular studies \cite{xia2019persistent, chen2023path, anand2022topological}.

These encodings, also called molecular descriptors or representations, are essential in cheminformatics. For example, molecular fingerprints are frequently used in chemical space mapping and virtual screening. In this section, we introduce models based on persistent Stanley--Reisner theory, which builds complexes from a molecule's structure and examines those complexes' topological and geometric characteristics. A key step in this process is filtering the resulting simplicial complex according to a physical parameter, often a radius, which helps reveal different layers of a molecule's shape and connectivity. This process is illustrated in Figure~\ref{fig:scheme}. We propose two applications of persistent Stanley--Reisner theory to molecular analysis.

	\subsection{Discrimination of Isomeric Structures}

	\begin{figure}[h!]
	\centering
	\subfigure{
		\includegraphics[width=0.99\textwidth]{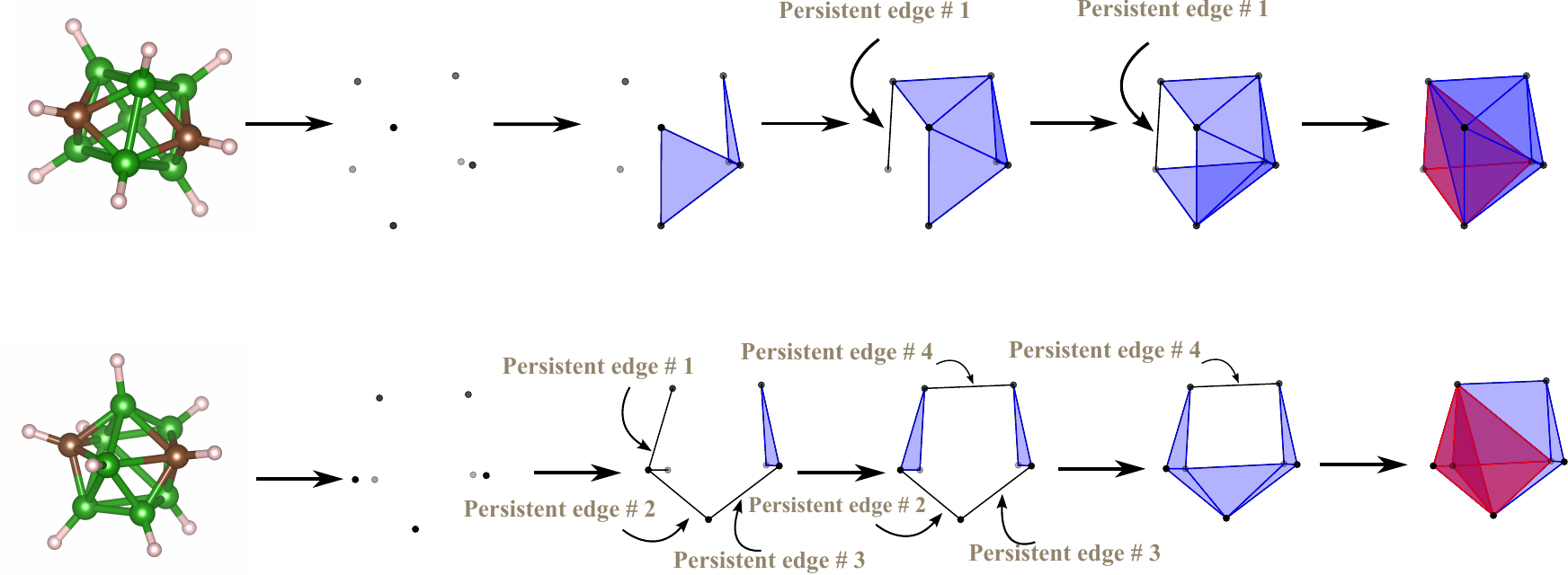} 
	}
	\caption{Structural representations and filtrations of two isomers. The top row presents the first isomer and its corresponding filtration, while the bottom row depicts the second and filtration processes. In both structures, hydrogen atoms are represented in white, boron atoms in green, and carbon atoms in brown.}
		\label{fig:isomers}
\end{figure}

	Organic molecules' structural and functional properties are fundamentally determined by their three-dimensional atomic configurations. In particular, the spatial arrangement of atoms within a molecule governs its physical and chemical characteristics, thereby giving rise to isomerism—a phenomenon in which distinct molecules possess an identical molecular formula yet differ in either their atomic connectivity or spatial orientation \cite{clayden2012organic, nelson2017lehninger, pauling1945nature}. Traditional methods for distinguishing isomers rely on the analysis of physical properties (e.g., melting point, boiling point, density, and refractive index), while computational techniques—such as graph-theoretic and topological methods—have also been developed for this purpose \cite{suwayyid2024persistentDirac, suwayyid2024persistentMayer}.

This study uses the facet persistence barcodes as a topological invariant to differentiate between two isomeric structures of \(\mathrm{B_7C_2H_9}\) and their associated protein complexes. The molecular systems under investigation (visualized using Schrödinger's Maestro software, see Figure \ref{fig:isomers}) are analyzed via a Vietoris-Rips filtration to extract salient topological features. Specifically, we consider two isomers of \(\mathrm{B_7C_2H_9}\) carborane, which consists of boron, carbon, and hydrogen atoms. In cases where the isomers exhibit similar atomic arrangements, an element-specific filtration can enhance discrimination; hence, we focus our filtration on the boron atoms, exploiting their spatial organization as a distinctive marker.

	Direct visual discrimination of these isomers is challenging. However, by varying the filtration parameter over the interval \(\left[0,3\right]\,\text{\AA}\), we observe significant differences in the dimension-1 facet persistence barcodes corresponding to edges: the first isomer exhibits a single persistent facet ideal corresponding to an edge, whereas the second isomer displays three or more prominent persistent facet ideals (see Figures~\ref{fig:isomers}). 
The fact that these two barcodes differ indicates that the persistent facet barcodes are sensitive to subtle and substantial differences in the geometric structures and thus provide a noteworthy mathematical basis for the differentiation of molecules. Figure~\ref{fig:isomers_barcodes} presents the Stanley-Reisner barcodes of the primes corresponding to the edges of the first and second isomers, respectively. Moreover, Figure \ref{fig:figure1} and Figure \ref{fig:figure2} show their diagrams respectively.
	
	\begin{figure}[h!]
		\centering
				\subfigure[]{
			\includegraphics[width=0.42\textwidth]{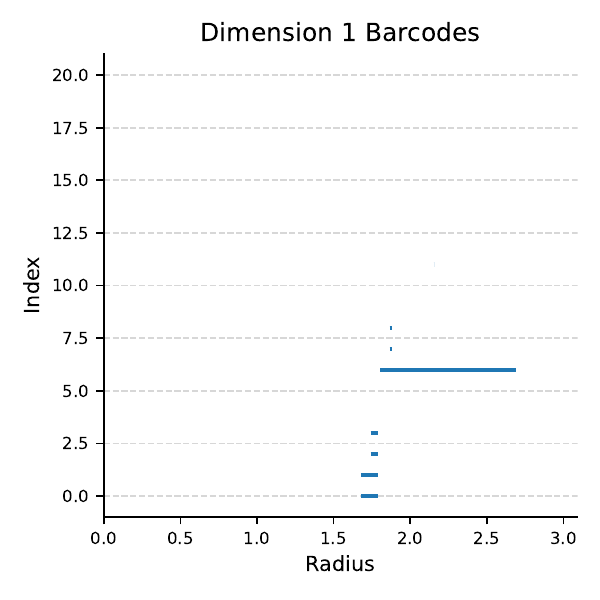} 
		}
\hspace{1cm}
				\subfigure[]{
			\includegraphics[width=0.42\textwidth]{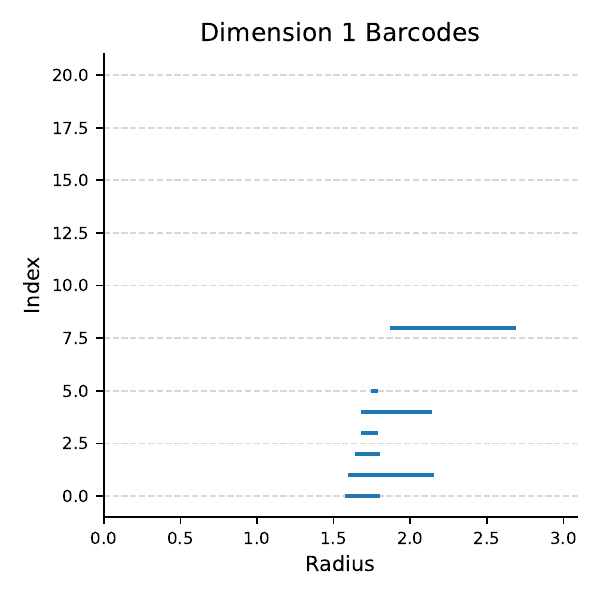} 
		}
		\caption{Facet persistence barcodes of the two isomers. (a)  the barcodes of the first isomer. (b) the barcodes of the second isomer. The first isomer exhibits a single essential persistent edge, whereas the second isomer features at least three or more persistent edges.}
				\label{fig:isomers_barcodes}
	\end{figure}
	
			\begin{figure}[h!]
			\centering
		\subfigure[]{
			\includegraphics[width=0.48\textwidth]{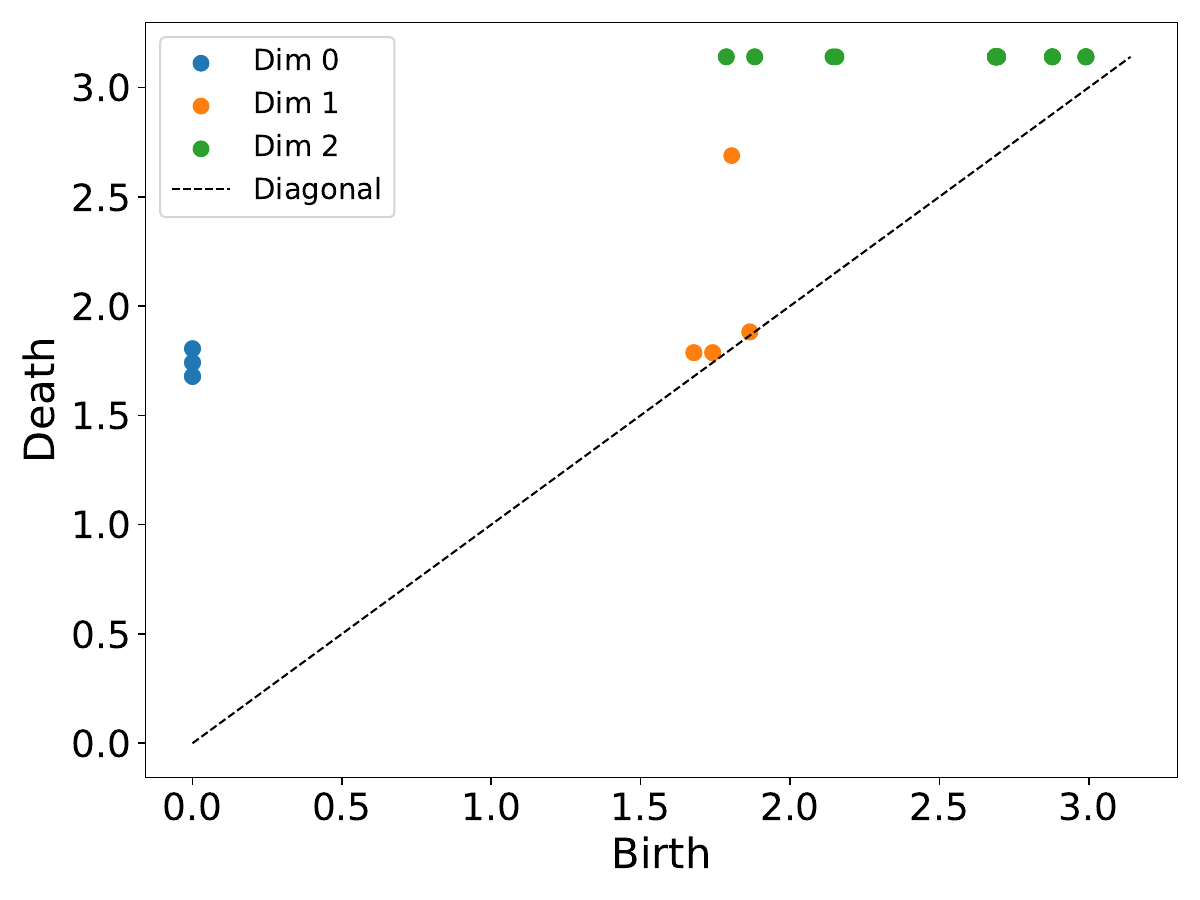} 
			\label{fig:figure1}
		}
		\subfigure[]{
			\includegraphics[width=0.48\textwidth]{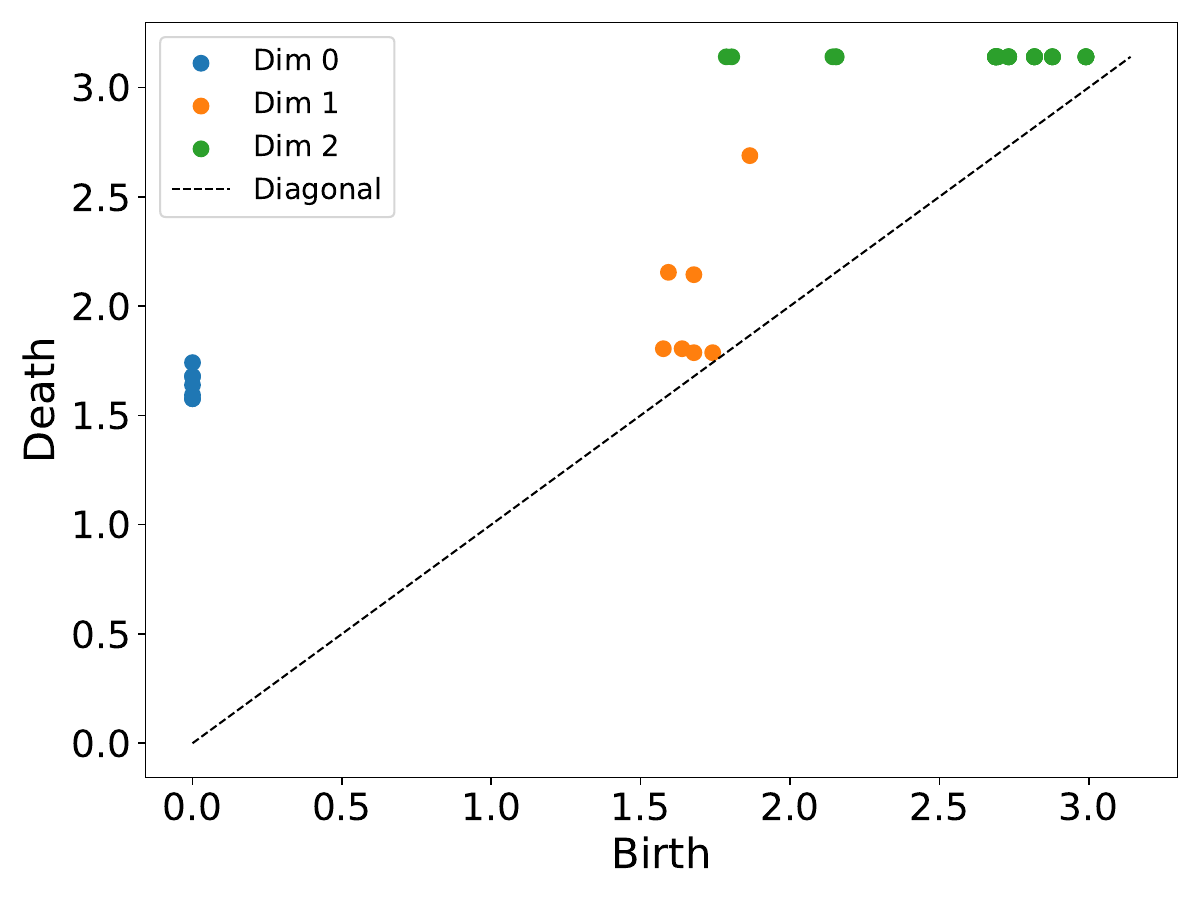} 
			\label{fig:figure2}
		}
		\caption{Facet persistence diagrams for the isomers under investigation. (a) the diagram corresponding to the first isomer, capturing the persistence of topological features across the first three homological dimensions. (b) the diagram corresponding to the second isomer, similarly capturing the persistence intervals of the faces of the first three dimensions.}
		\label{fig:isomers_diagrams}
	\end{figure}

	\subsection{Classification of Metal Halide Perovskites}
	
Metal halide perovskites, organic and inorganic, represent cost-effective, highly processable, and performance-efficient materials with significant potential in photovoltaic applications. These materials exhibit different crystal structures that undergo temperature-dependent phase transitions from cubic to tetragonal and orthorhombic configurations. 
Organic/inorganic metal halide perovskites (OIHPs) generally conform to the \(\mathrm{ABX_3}\) crystal structure, where \(A\) denotes a monovalent cation (either organic or inorganic, such as Cs or \(\mathrm{CH_3NH_3}\)), \(B\) represents a divalent cation (e.g., Pb, Sn), and \(X\) corresponds to a halide anion (e.g., Cl, I, Br). 
Topological data analysis (TDA) methods and machine learning algorithms have shown promising results in characterizing these materials \cite{wee2023persistent, suwayyid2024persistentMayer}. This study uses persistent Stanley-Reisner theory to tell apart the phase and atomic constitution of OIHP material. Significantly, this approach relies exclusively on geometric data, such as atomic positions, highlighting its effectiveness in characterizing OIHP materials.
The analysis focuses on three variants of Methylammonium lead halides (\(\mathrm{MAPbX_3}\)), where \(X\) represents Cl, Br, or I, across three distinct phases: orthorhombic, tetragonal, and cubic. Each variant consists of 300 samples, with 100 for each phase yielding 900 samples where the samples are obtained from \cite{wee2023persistent}.

	\begin{figure}[h!]
		\centering
		\subfigure{
			\includegraphics[width=0.65\textwidth]{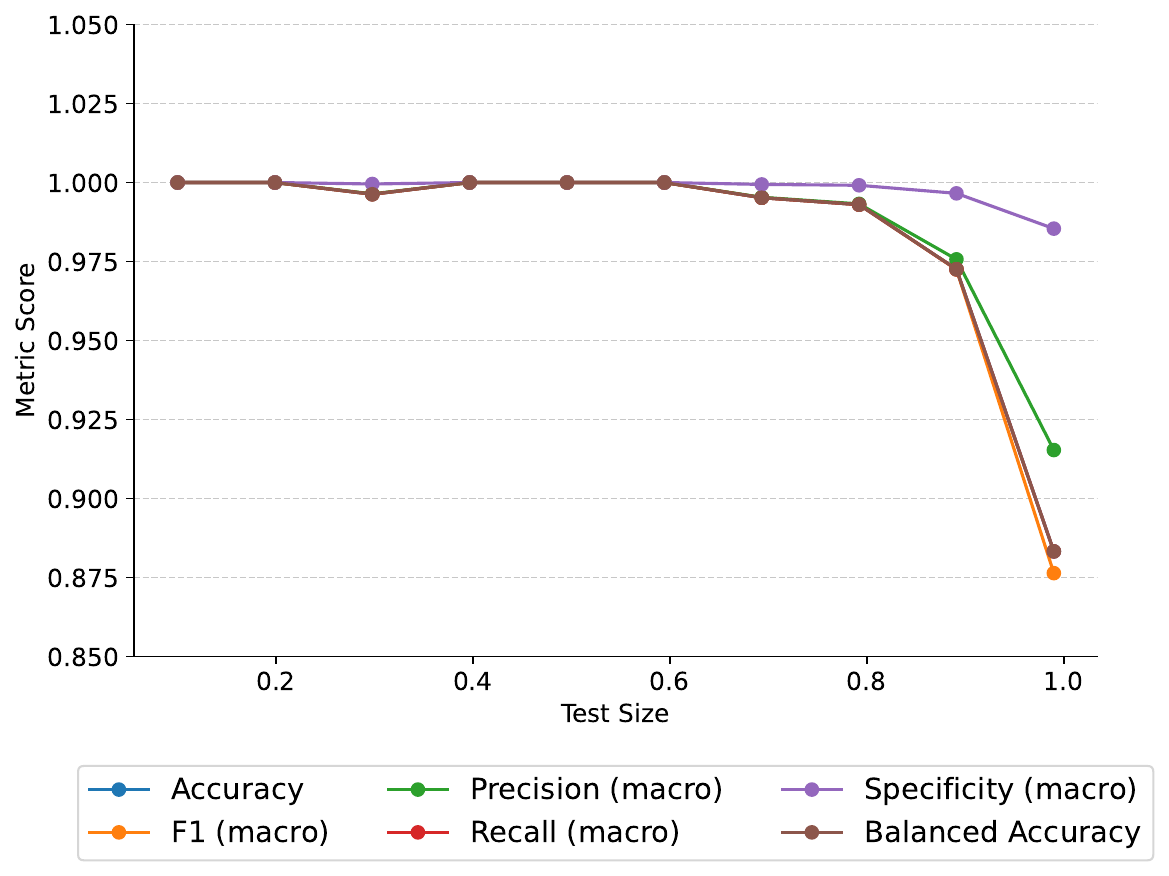} 
		}
		\caption{The performance of utilizing the critical values of the persistent facet ideals evaluated with six metrics in classifying OIHP materials. The 5-Nearest Neighbors (5-NN) classifier is used on the sets of the critical values generated from the Vietoris--Rips filtration. }
		
		\label{fig:app2_performance}
	\end{figure}
	In our methodology, atoms are treated as points in Euclidean space, and a sequence of Vietoris-Rips simplicial complexes is constructed. The Stanley-Reisner critical values are computed for each molecule, with a filtration radius ranging from \(0 \, \text{\AA}\) to \(7 \, \text{\AA}\). The Hausdorff distance is utilized to quantify differences between the associated sets of critical values. 
	To evaluate classification performance, robustness, and accuracy, we employ a 5-nearest neighbors (5NN) classifier and vary the test set size. As shown in Figure \ref{fig:app2_performance}, classification accuracy is perfect for small test sets and remains nearly perfect even when the test set size reaches 80\%. This result emphasizes the robustness and stability of the proposed methodology. 
	Overall, the study concludes that persistent Stanley-Reisner theory effectively captures the geometric information of OIHPs, providing a mathematically rigorous approach to their characterization. In a comparative analysis, we evaluated a persistent homology–based approach utilizing the persistent Betti numbers of the first two homological dimensions. Empirically, this method performed markedly poorly in the present application, exhibiting an accuracy below fifty percent. This is supported by Figure 5 (discrete Dirac) in \cite{wee2023persistent}.

\section{Conclusion}
Combinatorial commutative algebra integrates commutative algebra and combinatorics. One of its main subjects is the Stanley–Reisner ring of a simplicial complex. We introduce persistent Stanley-Reisner theory by combining the classical apparatus of Stanley--Reisner rings with key ideas from persistent homology. By extending Hochster’s formula to the parametrized setting, we can track how combinatorial and algebraic properties of a simplicial complex evolve as the filtration parameter changes. We showed that the births and deaths of the persistent facet ideals, which correspond to facets in the filtration, produce barcodes analog to those arising in standard topological persistence.

Our stability results guarantee that small perturbations in the underlying filtration do not drastically alter the associated facet persistence barcodes, mirroring the well-known robustness of persistent homology. Furthermore, explicit constructions were given for persistent graded Betti numbers, persistent $h$-vectors, and persistent $f$-vectors, establishing a concrete link between discrete geometry, commutative algebra, and topological data analysis. 

We use two applications to demonstrate the potential of this new theory for data science. In the first application, we used persistent Stanley--Reisner invariants to distinguish molecular isomers. This demonstrates that even subtle three-dimensional (3D) atomic arrangement differences can be encoded and captured in persistence barcodes. In the second, we deployed these methods to classify organic-inorganic metal halide perovskites, leveraging purely geometric information (atomic positions) to highly accurately separate phase and compositional variants.

We anticipate this new framework will find additional applications involving intrinsically complex and high-dimensional data in various fields. Future directions include extending the theory to multi-parameter filtrations, investigating computational algorithms that scale to large data sets, exploring new algebraic topological invariants derivable from persistent Stanley-Reisner theory, and studying commutative algebra with different topological spaces with filtrations, such as path complex, directed flag complex, cellular sheaves, graphs, and hypergraphs \cite{ha2008monomial}.

\section*{Acknowledgments}
This work was supported in part by NIH grants   R01AI164266, and R35GM148196, NSF grants DMS-2052983,  DMS-2245903,  and IIS-1900473,  MSU Foundation, and Bristol-Myers Squibb 65109.
F.S. thanks King Fahd University of Petroleum and Minerals for their support.
G.W.W. thanks Dr. T\`{a}i Huy H\`{a} for useful discussions.  

 \printindex
		
\bibliographystyle{unsrt}
\bibliography{references}
\end{document}